\documentclass[]{interact}

\usepackage{enumerate}
\usepackage{epstopdf}
\usepackage[caption=false]{subfig}
\usepackage[numbers,sort&compress]{natbib}

\makeatletter
\def\NAT@def@citea{\def\@citea{\NAT@separator}}
\makeatother

\makeatletter
\renewcommand*\env@matrix[1][*\c@MaxMatrixCols c]{%
  \hskip -\arraycolsep
  \let\@ifnextchar\new@ifnextchar
  \array{#1}}
\makeatother

\theoremstyle{plain} 
\newtheorem{theorem}{Theorem}[section]
\newtheorem{corollary}[theorem]{Corollary}
\newtheorem{lem}[theorem]{Lemma}

\theoremstyle{definition} 
\newtheorem{definition}[theorem]{Definition}
\newtheorem{remark}[theorem]{Remark}

\newtheorem{assumption}[theorem]{Assumption}
\newtheorem*{assumption*}{Assumption}

\usepackage{comment}
\usepackage{arydshln}
\usepackage{soul}
\usepackage[color=yellow!30, linecolor=orange, colorinlistoftodos]{todonotes}
\usepackage[breakable]{tcolorbox}
\tcbuselibrary{theorems}

\colorlet{mycyan}{cyan!20}
\colorlet{myorange}{orange!40}

\colorlet{myyellow}{yellow!40}
\sethlcolor{myyellow}

\newcommand{\dstarinvodd}{d^{(2n+1)^{*-1} } }
\newcommand{\dstarinveven}{d^{(2n)^{*-1} } }
\newcommand{\define}[1]{\emph{#1}}

\usepackage{amsmath}
\usepackage{amssymb}
\usepackage{abraces}
\usepackage[utopia]{mathdesign}
\usepackage{hyperref}
\usepackage{tikz}
\usetikzlibrary{matrix, calc}
\usetikzlibrary{positioning}
\usepackage{datetime}

\allowdisplaybreaks

\newcommand{\R}{\mathbb R}
\newcommand{\C}{\mathbb C}
\newcommand{\N}{\mathbb N}

\makeatletter
\def\widebreve{\mathpalette\wide@breve}
\def\wide@breve#1#2{\sbox\z@{$#1#2$}%
     \mathop{\vbox{\m@th\ialign{##\crcr
\kern0.08em\brevefill#1{0.8\wd\z@}\crcr\noalign{\nointerlineskip}%
                    $\hss#1#2\hss$\crcr}}}\limits}
\def\brevefill#1#2{$\m@th\sbox\tw@{$#1($}%
  \hss\resizebox{#2}{\wd\tw@}{\rotatebox[origin=c]{90}{\upshape(}}\hss$}

\begin{document}

\title{Explicit relation between two resolvent matrices of the truncated Hausdorff
matrix moment problem}

\author{
\name{Abdon E. Choque-Rivero\textsuperscript{a}\thanks{
CONTACT A.~E.  Choque-Rivero.  Email: abdon@ifm.umich.mx} 
and Monika Winklmeier\textsuperscript{b}
\thanks{CONTACT M. Winklmeier. Email: mwinklme@uniandes.edu.co}}
\affil{\textsuperscript{a} Universidad Michoacana de San Nicol\'as de Hidalgo, 
Instituto de F\'isica y Matem\'aticas,
Morelia, Mich., M\'exico;
\textsuperscript{b}~Universidad de los Andes, Departamento de Matem\'aticas, Bogot\'a, Colombia.
}
}

\maketitle

\begin{abstract}
   We consider the explicit relation between two resolvent matrices related to the truncated Hausdorff matrix moment problem (THMM) in the case of an even and odd number of moments.
   This relation is described with the help of four families of orthogonal matrix polynomials on the finite interval $[a,b]$ and their associated second kind polynomials.
\end{abstract}

\begin{amscode}30E05, 42C05,  47A56.  \end{amscode}

\section{Introduction}

In the classical scalar Hamburger moment problem, a real sequence $(s_j)_{j=0}^\infty$ is given and the problem is to find a measure $\sigma$ such that 
\begin{align*}
   s_j = \int_\R t^j \sigma(dt),
   \qquad j\in \N_0.
\end{align*}
The problem is called \define{determinate} if there is exactly one solution $\sigma$; it is called \define{indeterminate} if there is more than one solution.
In the indeterminate case, all solutions can be described by a $2\times 2$ Nevanlinna matrix \cite[Definition 2.4.3]{akh} whose elements are transcendental functions, see \cite[Sec. 3.4, p. 113]{akh}.

In this paper we are concerned with the \define{truncated Hausdorff matrix moment} (THMM) problem.
Let $a < b$ be real numbers, $q, m \in\N$ and let $(s_j)_{j=0}^m$ be hermitian $q\times q$ matrices.
The truncated Hausdorff matrix moment on $[a,b]$ problem consists in finding all positive matrix measures
$\sigma$ whose $j$th moment is equal to $s_j$, that is, 
\begin{align*}
   s_j = \int_a^b t^j \sigma(dt),
   \qquad j=0,\dots, m.
\end{align*}
We denote the set of all solutions $\sigma$ of the THMM by 
\begin{align}
   \label{eq:solutions}
   {\mathcal M}_\geq^q[[a,b],\mathfrak B\cap[a,b];\ (s_j)_{j=0}^m].
\end{align}
The Stieltjes transform
\begin{align*}
   s(z) = \int_a^b \frac{1}{t-z}\, \sigma(dt)
\end{align*}
shows that the set \eqref{eq:solutions} is equivalent to the so-called \define{associated solution set}
\begin{align}
   \label{eq:asssolutions}
   {\mathfrak S}_\geq^q[[a,b],\mathfrak B\cap[a,b];\ (s_j)_{j=0}^{m}]
   :=\left\{
   s(z)=\int_{[a,b]}\frac{d\sigma(t)}{t-z},\,
   \sigma\in {\mathcal M}_\geq^q[[a,b],\mathfrak B\cap[a,b];
   (s_j)_{j=0}^{m}]
   \right\}.
\end{align}
As in the scalar case, it can be shown that the set \eqref{eq:asssolutions} is given by all functions $s$ of the form
\begin{equation}
   s(z)=
   \left[\alpha^{(m)}(z)\,{\bf p}(z)+\beta^{(m)}(z)\,{\bf q}(z)\right]
   \left[\gamma^{(m)}(z)\,{\bf p}(z)+\delta^{(m)}(z)\,{\bf q}(z)\right]^{-1}, %
   \label{0002a}
\end{equation}
where  ${\bf p}$ and ${\bf q}$ are $q\times  q$ matrix-valued functions of $z$ which are meromorphic in $\C\setminus [a,b]$ and satisfy certain positivity conditions, see for instance \cite{D-C,abdon1, abdon2}.
The coefficient matrix 
\begin{equation}
   \label{0003a}
   \begin{pmatrix}
      \alpha^{(m)}(z) & \beta^{(m)}(z) \\
      \gamma^{(m)}(z) & \delta^{(m)}(z)
   \end{pmatrix}
\end{equation}
is called a \define{resolvent matrix} of the THMM problem.
Note that the matrix $U^{(m)}$ is not unique.
Indeed, the first resolvent matrix was found in 2001 in \cite[Equality (10) and Equality (30)]{D-C}
under the assumption that certain matrices constructed from the given moments are strictly positive.
We will call it the \define{resolvent matrix with respect to the point $0$}.
Later, in 2006 and 2007, other resolvent matrices were found in \cite[Equality (6.20)]{abdon1} 
for even $m$ and in \cite[Equalities (6.54) through (6.57)]{abdon2} for odd $m$.
We will call these matrices the \define{resolvent matrices with respect to the point $a$},
see \eqref{RM2nm1} and \eqref{RM1}.
The papers \cite{abdon1} (2006), Equality (6.20) and \cite{abdon2}
(2007), Equalities (6.54) through (6.57)
contain necessary and sufficient conditions for the THMM problem to be indeterminate.
In these three works the entries of the resolvent matrices are polynomials in $z$.
In 2015,
the resolvent matrices from \cite{abdon1, abdon2}
were expressed in terms of four families of orthogonal polynomials and their associated second kind polynomials in \cite[Equalities (3.24)-odd, (3.37)-even]{abKN}.

We remark that the scalar version of the Hausdorff moment problem was studied by Krein and Nudel'man in their book \cite[Page 115]{KreinNudelman}.

\subsection*{Matrices of moments}

Let $(s_j)_{j=0}^m$ be a sequence of hermitian $q\times q$ matrices and let $\sigma\in \mathcal M_\geq^q[[a,b],\mathfrak B\cap[a,b];\ (s_j)_{j=0}^m]$ be a solution of the THMM problem.
We will need the following perturbed measures defined on Borel sets $B$ by
\begin{align*}
   \sigma_{2}(B):=&\int_{B}(b-t)(t-a)\, d\sigma(t), \\
   \sigma_3(B):=&\int_{B}(b-t)\, d\sigma(t),\\
   \sigma_4(B):=&\int_{B}(t-a)\, d\sigma(t)
\end{align*}
and their corresponding sequence of moments $(s^{(r)}_j)_{j=0}^m$, $r=1,2,3,4$.
Clearly they satisfy
\begin{alignat*}{3}
   s_j^{(1)}   & = s_j, \\
   s_{j}^{(2)} & = -abs_{j}+(a+b)s_{j+1}-s_{j+2},
   \\
   s_j^{(3)} & = bs_{j}-s_{j+1}, 
   \\
   s_j^{(4)} & = -as_{j}+s_{j+1}. 
\end{alignat*}
Note that $s_j^{(2)} = b s_j^{(4)} - s_{j+1}^{(4)} = -a s_j^{(3)} + s_{j+1}^{(3)}$.

For $r=1,2,3,4$, we define the block Hankel matrices $H_{r,j}$ by
\begin{align}
   H_{r,j}:=
   (s^{(r)}_{k+\ell})_{\ell,k=0}^j =
   &
   \begin{tikzpicture}[baseline=(current bounding box.center),
      every left delimiter/.style={xshift=.5em},
      every right delimiter/.style={xshift=-.5em},
      ampersand replacement=\&,
      ]
      \matrix (m) [matrix of math nodes,nodes in empty cells, right delimiter={)},left delimiter={(} ]{
      s_0^{(r)} \& s_1^{(r)}\vphantom{s_{j}^{(r)}} \&  \hspace*{4ex} \& s_{j}^{(r)}  \\[-1ex]
      s_1^{(r)} \& s_2^{(r)}\vphantom{s_{j}^{(r)}} \& \& s_{j+1}^{(r)}  \\
      \&       \&  \\[1ex] 
      s_j^{(r)} \& s_{j+1}^{(r)} \& \& s_{2j}^{(r)}  \\
      } ;
%
      \draw[thick, dotted] ($(m-2-4.south west)+(0ex,0ex)$)-- ($(m-4-2.north east)+(0ex,-0ex)$);
      \draw[thick, dotted] ($(m-1-2.east)+(0ex,0ex)$)-- ($(m-1-4.west)+(0ex,-0ex)$);
      \draw[thick, dotted] ($(m-2-2.east)+(0ex,0ex)$)-- ($(m-2-4.west)+(0ex,-0ex)$);
      \draw[thick, dotted] ($(m-4-2.east)+(0ex,0ex)$)-- ($(m-4-4.west)+(0ex,-0ex)$);
      \draw[thick, dotted] ($(m-2-1.south)+(0ex,0ex)$)-- ($(m-4-1.north)+(0ex,-0ex)$);
      \draw[thick, dotted] ($(m-2-2.south)+(0ex,0ex)$)-- ($(m-4-2.north)+(0ex,-0ex)$);
      \draw[thick, dotted] ($(m-2-4.south)+(0ex,0ex)$)-- ($(m-4-4.north)+(0ex,-0ex)$);
   \end{tikzpicture}.
   \label{201A}
\end{align}
In the case of an odd number of moments, i.e., if $m=2n$, the THMM has a solution if the block matrices $H_{1,n}$ and $H_{2,n-1}$ are both nonnegative. See \cite[Theorem 1.3]{abdon2}.
In the case of an even number of moments, i.e., if $m=2n+1$, the THMM has a solution if the block matrices $H_{3,n}$ and $H_{4,n}$ are both nonnegative.
See \cite[Theorem 1.3]{abdon1}.
\medskip

We will also need the following matrices.
\begin{equation}\label{eqHt1n}
   \widetilde H_{r,j}:=(s_{k+\ell+1}^{(r)})_{\ell,k=0}^j
  =
   \begin{tikzpicture}[baseline=(current bounding box.center),
      every left delimiter/.style={xshift=.5em},
      every right delimiter/.style={xshift=-.5em},
      ampersand replacement=\&,
      ]
      \matrix (m) [matrix of math nodes,nodes in empty cells, right delimiter={)},left delimiter={(} ]{
      s_1^{(r)} \& s_2^{(r)} \vphantom{s_{j+1}^{(r)}}\&  \hspace*{4ex} \& s_{j+1}^{(r)}  \\[-1ex]
      s_2^{(r)} \& s_3^{(r)} \vphantom{s_{j+1}^{(r)}}\& \& s_{j+2}^{(r)}  \\
      \&       \&  \\[1ex] 
      s_{j+1}^{(r)} \& s_{j+2}^{(r)} \& \& s_{2j+1}^{(r)}  \\
      } ;
%
      \draw[thick, dotted] ($(m-2-4.south west)+(0ex,0ex)$)-- ($(m-4-2.north east)+(0ex,-0ex)$);
      \draw[thick, dotted] ($(m-1-2.east)+(0ex,0ex)$)-- ($(m-1-4.west)+(0ex,-0ex)$);
      \draw[thick, dotted] ($(m-2-2.east)+(0ex,0ex)$)-- ($(m-2-4.west)+(0ex,-0ex)$);
      \draw[thick, dotted] ($(m-4-2.east)+(0ex,0ex)$)-- ($(m-4-4.west)+(0ex,-0ex)$);
      \draw[thick, dotted] ($(m-2-1.south)+(0ex,0ex)$)-- ($(m-4-1.north)+(0ex,-0ex)$);
      \draw[thick, dotted] ($(m-2-2.south)+(0ex,0ex)$)-- ($(m-4-2.north)+(0ex,-0ex)$);
      \draw[thick, dotted] ($(m-2-4.south)+(0ex,0ex)$)-- ($(m-4-4.north)+(0ex,-0ex)$);
   \end{tikzpicture}
\end{equation}
obtained from $H_{r,j+1}$ by deleting the first row and the last column (or alternatively, by deleting the first column and the last row).

\begin{definition} \label{deboth}  Let $[a,b]$ be a finite interval on the real axis
   $\R$.
   The sequence of $q\times q$ hermitian matrices $(s_k)_{k=0}^{2n}$ (resp. $(s_k)_{k=0}^{2n+1}$)
   is called a \define{Hausdorff positive definite sequence on $[a,b]$}
   if the block Hankel matrices $H_{1,n}$ and $H_{2,n-1}$
   (resp. $H_{3,n}$ and $H_{4,n}$)  are both positive
   definite matrices.
   If the sequence $(s_k)_{k=0}^{m}$ 
   is a positive definite
   sequence, then the THMM problem is called \define{non degenerate}.
\end{definition}

In the present work, we deal with the non degenerate case.
In this case, it can be shown that 
a holomorphic function $s$ defined on the upper half plane is an associated solution of the THMM if and only if certain block operator matrices involving $s$ and the given moments are nonnegative, 
see \cite{D-C} and \cite{abdon1, abdon2}.
Factorization of these matrices shows that this is the case if and only if $s$ is of the form~\eqref{0002a}.
From \eqref{0002a}, we clearly see that we obtain the set all of solutions of the THMM problem if and only if we know the so-called resolvent matrix~\eqref{0003a}
together with the allowed 
pairs ${\bf p}$ and ${\bf q}$.
See \cite[Definition 5.2]{abdon1} and \cite[Definition 5.2]{abdon2}. 
Different representations of the resolvent matrix are available.

\subsection*{Resolvent matrix with respect to the point $0$}
The techniques employed in \cite{D-C} and in \cite{abdon1, abdon2} show that the resolvent matrix can be written as \eqref{eqn:VV14m1} for an even number of moments and as
\eqref{eqn:VV14} for an odd number of moments
of the resolvent matrix.
We call them the \define{resolvent matrices with respect to the point $0$} and we will denote them by $V^{(m)}$.
Its entries $\alpha^{(m)}(z)$, $\beta^{(m)}(z)$, $\gamma^{(m)}(z)$, and $\delta^{(m)}(z)$
are matrix valued polynomial functions of $z$ which are uniquely constructed from the input set of moments $(s_j)_{j=0}^m$.

Its expansion in powers of $z$ is
\begin{align}
   \label{v2nm1A}
   V^{(2n+1)} &= \widetilde A_0+ z\widetilde A_1+\ldots+
   z^{n+1}\widetilde A_{n+1}+z^{n+2}\widetilde A_{n+2},
   &&\text{if }\ m=2n+1,
   \\
   \label{v2nAA}
   V^{(2n)} &= \widetilde B_0+ z\widetilde B_1+\ldots+
   z^{n}\widetilde B_{n}+z^{n+1}\widetilde B_{n+1},
   &&\text{if }\ m=2n.
\end{align}
The coefficients $\widetilde A_j$ and $\widetilde B_j$ are rational functions in
$a$ and $b$, see Remarks~\ref{remsep1} and \ref{remsep2} for explicit formulas.

\subsection*{Resolvent matrix with respect to the point $a$}
In \cite{abKN} a different approach was used to describe the resolvent matrix. 
The first author used the resolvent matrix from \cite{abpol} and expressed it in terms of orthogonal polynomials of the first and the second kind.
This leads to the resolvent matrix 
\eqref{RM2nm1} in the case of an even number of moments and to 
\eqref{RM1} in the case of an odd number of moments denoted by $U^{(m)}$.
We call them the \define{resolvent matrices with respect to the point $a$}.
\medskip

For $m=2n+1$ the resolvent matrix 
$U^{(2n+1)}$ can be written as
\begin{equation}\label{v2nm1C}
U^{(2n+1)}(z)=\widetilde C_0+ (z-a)\widetilde C_1+\ldots+
(z-a)^{n+1}\widetilde C_{n+1}+(z-a)^{n+2}\widetilde C_{n+2}.
\end{equation}
The coefficients $\widetilde C_j$ for $0\leq j\leq n+2$ are rational functions in $a$ and $b$;
see Remark~\ref{remsep3} for explicit formulas.

For $m=2n$ the resolvent matrix 
$U^{(2n)}$ can be expressed in the following form
\begin{equation}\label{v2nD}
U^{(2n)}(z)=\widetilde D_0+ (z-a)\widetilde D_1+\ldots+
(z-a)^{n+1}\widetilde D_{n+1}.
\end{equation}
The coefficients $\widetilde D_j$ for $0\leq j\leq n+1$ are rational functions in $a$ and $b$;
see Remark~\ref{remsep4} for explicit formulas.
\medskip

The particular representation of the resolvent matrix of the THMM problem plays a crucial role in its factorization.
Different factorizations of the mentioned resolvent matrices lead to four families of orthogonal matrix polynomials, the Dyukarev-Stieltjes parameters, continued fractions and the coefficients of the three term recurrence relation. 
Consequently, the explicit relation between the resolvent matrices given in
\cite{D-C} and \cite{abKN} 
is relevant.
\smallskip

To the best of the knowledge of the authors, the matrix moment problem was first studied in \cite{k1949} and \cite{k1949a}. In \cite{simon}, \cite{dette2}, \cite{duranker}, \cite{duran2}, \cite{duran1}, \cite{duran3}, \cite{gero}  orthogonal matrix polynomials were considered  to solve matrix moment problems or structural formulas and differential relations.
 In \cite{DFK}, \cite{dyu0}, \cite{D-C2}, \cite{D-C3} V.P.~Potapov's method of
 matrix inequalities was used to solve interpolation problems in certain
 classes of functions.
 In \cite{fkm1}, the THMM problem was recently studied via a
Schur-Nevanlinna type algorithm.  The operator approach was applied to solve the THMM problem in \cite{zagor} and \cite{aboperator}.

\subsection*{Main result of this work.}
In this work we give an explicit relation between the two different types of resolvent matrices in the form
\begin{equation*}
   U^{(m)}(z) = C V^{(m)}(z) D.
\end{equation*}
The importance of the indicated explicit relation is twofold: 
Firstly, it allows us to find new relations among the matrix orthogonal polynomials, their second kind polynomials and block elements of the resolvent matrices appearing in  \cite{D-C}, \cite{abKN}.
For instance, the identities in Corollary~\ref{rem00BB} are new.
Secondly, with such an explicit relation important new expressions for the Dyukarev-Stieltjes parameters, continued fractions can be obtained.
From the point of view of applications, the found relation between two resolvent matrices would allow to rewrite and develop the representation of the set of admissible controls for bounded control systems; 
see \cite{ieeeA1}, \cite{optimo}, \cite{cks2010}.

In our work we have to distinguish the cases when $m$ is even or odd. 
We will present various explicit formulas and identities which exhibit interactions between the matrices corresponding to odd and even values of $m$; 
see Section~\ref{sec0003}.

\section{Matrices of moments and orthogonal matrix polynomials} \label{sec0002}
In this section, we reproduce some notation from \cite{abcoef} that appear throughout this work.
In particular, we recall the definition of the orthogonal matrix polynomials (OMP)
$P_{k,j}$ on $[a,b]$ and their second kind polynomials $Q_{k,j}$.

\subsection{Shifts and truncations} 
We start with some auxiliary matrices which do not depend on the moments.
Let $q\in\N$ and $\mathbb M_q(k\times \ell)$ be the set of all $k\times \ell$ block matrices whose entries are $q\times q$ matrices.
We define the ``block down shift'' on the vector space $\mathbb C_q^{j+1}$
\begin{align}\label{eqTn}
   T_{0}:=0,\quad
   T_{j}:= 
   \tikzset{baseline=(current bounding box.center),
   every left delimiter/.style={xshift=.5em},
   every right delimiter/.style={xshift=-.5em},
   column sep=.5ex,
   row sep=-.5ex,
   }
   \begin{tikzpicture}[baseline=(current bounding box.center),
      every left delimiter/.style={xshift=.5em},
      every right delimiter/.style={xshift=-.5em},
      ampersand replacement=\&,
      ]
      \matrix (m) [matrix of math nodes,nodes in empty cells, right delimiter={)},left delimiter={(} ]{
      0 \&       \&   \&  0\\ 
      I \&       \&   \&   \\ 
      \& \rule{1ex}{0ex}  \&   \&   \\[2ex]
      0 \&       \& I \&  0 \\
      } ;
%
      \draw[thick, dotted] ($(m-1-1.south east)+(0ex,0ex)$)-- ($(m-4-4.north west)+(0ex,-0ex)$);
      \draw[thick, dotted] ($(m-2-1.south east)+(0ex,0ex)$)-- ($(m-4-3.north west)+(0ex,-0ex)$);
   \end{tikzpicture}
   \in \mathbb M_q((j+1)\times (j+1))
   ,\qquad j\geq 1.
\end{align}
Note that all entries in $T_j$ are $q\times q$ matrices and all other entries are $0$.

Clearly, every $T_j$ is nilpotent, and therefore the matrix valued function
\begin{align}
   R_{j}:\mathbb{C}\to \mathbb M_q((j+1)\times (j+1)),\quad
   \label{52}R_{j}(z)&:=(I_{(j+1)q}-zT_{j})^{-1},\quad j\geq 0,
\end{align}
is well-defined and satisfies
$R_j(z) = \sum_{\ell=0}^j z^\ell T^\ell$.
\smallskip

We will also make use of the following $(j+1)\times j$ block matrices
\begin{align}
   L_{1,j}:=
   \begin{pmatrix}
      0  \\
      I_{jq\times jq}
   \end{pmatrix}
   =
   \begin{tikzpicture}[baseline=(current bounding box.center),
      every left delimiter/.style={xshift=.5em},
      every right delimiter/.style={xshift=-.5em},
      ampersand replacement=\&,
      ]
      \matrix (m) [matrix of math nodes,nodes in empty cells, right delimiter={)},left delimiter={(} ]{
      0 \&       \& 0  \\ 
      I \&  \hspace*{2ex} \& 0  \\ [3ex]
      0 \&       \& I  \\
      } ;
%
      \draw[thick, dotted] ($(m-2-1.south east)+(0ex,0ex)$)-- ($(m-3-3.north west)+(0ex,-0ex)$);
      \draw[thick, dotted] ($(m-1-1.east)+(0ex,0ex)$)-- ($(m-1-3.west)+(0ex,-0ex)$);
      \draw[thick, dotted] ($(m-2-1.east)+(0ex,0ex)$)-- ($(m-2-3.west)+(0ex,-0ex)$);
      \draw[thick, dotted] ($(m-3-1.east)+(0ex,0ex)$)-- ($(m-3-3.west)+(0ex,-0ex)$);
      \draw[thick, dotted] ($(m-2-1.south)+(0ex,0ex)$)-- ($(m-3-1.north)+(0ex,-0ex)$);
      \draw[thick, dotted] ($(m-2-3.south)+(0ex,0ex)$)-- ($(m-3-3.north)+(0ex,-0ex)$);
   \end{tikzpicture},
   \qquad
   L_{2,j}:=
   \begin{pmatrix}
      I_{jq\times jq} \\
      0
   \end{pmatrix}
   =
   \begin{tikzpicture}[baseline=(current bounding box.center),
      every left delimiter/.style={xshift=.5em},
      every right delimiter/.style={xshift=-.5em},
      ampersand replacement=\&,
      ]
      \matrix (m) [matrix of math nodes,nodes in empty cells, right delimiter={)},left delimiter={(} ]{
      I \&  \hspace*{2ex} \& 0  \\ [3ex]
      0 \&       \& I  \\
      0 \&       \& 0  \\ 
      } ;
%
      \draw[thick, dotted] ($(m-1-1.south east)+(0ex,0ex)$)-- ($(m-2-3.north west)+(0ex,-0ex)$);
      \draw[thick, dotted] ($(m-1-1.east)+(0ex,0ex)$)-- ($(m-1-3.west)+(0ex,-0ex)$);
      \draw[thick, dotted] ($(m-2-1.east)+(0ex,0ex)$)-- ($(m-2-3.west)+(0ex,-0ex)$);
      \draw[thick, dotted] ($(m-3-1.east)+(0ex,0ex)$)-- ($(m-3-3.west)+(0ex,-0ex)$);
      \draw[thick, dotted] ($(m-1-1.south)+(0ex,0ex)$)-- ($(m-2-1.north)+(0ex,-0ex)$);
      \draw[thick, dotted] ($(m-1-3.south)+(0ex,0ex)$)-- ($(m-2-3.north)+(0ex,-0ex)$);
   \end{tikzpicture}
   \in \mathbb M_q( (j+1)\times j)
   \label{82g}
\end{align}
and the block vectors
\begin{equation}
   v_{0}:=I,\quad
   v_{j}:=
   L_{2,j-1} v_{j-1}
   =
   \begin{pmatrix}
      I \\
      0_{jq\times q}
   \end{pmatrix},
   \quad
   j\geq1. \label{59}
\end{equation}
Clearly, the following relations hold.
\begin{alignat*}{7}
   L_{1,j} \begin{pmatrix} a_{1} \\ \vdots \\ a_{j}
   \end{pmatrix}
   &=
   \begin{pmatrix} 0  \\ a_{1} \\ \vdots \\ a_{j} \end{pmatrix},
   &
   \qquad
   & L_{2,j} \begin{pmatrix} a_{1} \\ \vdots \\ a_{j}
   \end{pmatrix}
   =
   \begin{pmatrix} a_{1} \\ \vdots  \\ a_{j} \\ 0
   \end{pmatrix},
   &
   \qquad
   & T_{j} \begin{pmatrix} a_{0} \\ a_{1} \\ \vdots \\ a_{j}
   \end{pmatrix}
   =
   \begin{pmatrix} 0 \\ a_{0} \\ \vdots  \\ a_{j-1}
   \end{pmatrix},
   \\[1ex]
   L_{1,j}^* \begin{pmatrix} a_0 \\ a_{1} \\ \vdots \\ a_{j}
   \end{pmatrix}
   &=
   \begin{pmatrix} a_{1} \\ \vdots \\ a_{j}
   \end{pmatrix},
   && L_{2,j}^* \begin{pmatrix} a_0 \\ \vdots \\ a_{j-1}  \\ a_{j}
   \end{pmatrix}
   =
   \begin{pmatrix} a_{0} \\ \vdots  \\ a_{j-1}
   \end{pmatrix},
   && T_{j}^* \begin{pmatrix} a_{0} \\ a_{1} \\ \vdots \\ a_{j}
   \end{pmatrix}
   =
   \begin{pmatrix} a_{1} \\ \vdots  \\ a_{j} \\ 0
   \end{pmatrix}
\end{alignat*}
where the $a_\ell$ can be $q\times q$ matrices or row vectors consisting of $q\times q$ matrices.
By transposition, we obtain
\begin{alignat*}{7}
   (a_0,\, \dots,\, a_{j}) L_{1,j} &= (a_1,\, \dots,\, a_{j}),
   &
   \qquad
   (a_0,\, \dots,\, a_{j-1}) L_{1,j}^* &= (0,\, a_0,\, \dots,\, a_{j-1}),
   \\
   (a_0,\, \dots,\, a_{j}) L_{2,j} &= (a_0,\, \dots,\, a_{j-1}),
   &
   (a_0,\, \dots,\, a_{j-1}) L_{2,j}^* &= (a_0,\, \dots,\, a_{j-1},\, 0),
   \\
   (a_0,\, \dots,\, a_{j}) T_{j} &= (a_1,\, \dots,\, a_{j},\, 0),
   &
   (a_0,\, \dots,\, a_{j}) T_{j}^* &= (0,\, a_0,\, \dots,\, a_{j-1}).
\end{alignat*}
It follows that 
\begin{gather}
   \label{eq:LLId}
   L_{1,j}^*L_{1,j} = L_{2,j}^*L_{2,j} = I,\\
   \label{eq:LLT}
   L_{1,j}L_{2,j}^* = T_j,\quad L_{2,j}^*L_{1,j} = T_{j-1}.
\end{gather}
If we multiply the first equality in \eqref{eq:LLT} by $L_{1,j}^*$ from the left and by $L_{2,j}$ from the right, we obtain with \eqref{eq:LLId}
\begin{gather}
   \label{eq:LTL1}
   L_{1,j}^*T_j = L_{2,j}^*, \quad
   T_{j}L_{2,j} = L_{1,j},\quad
   L_{1,j}^*T_{j}L_{2,j} = I.
\end{gather}
If we multiply the first equality in \eqref{eq:LTL1} by $L_{1,j}$ from the right and the second equality by $L_{2,j}^*$ from the left, we obtain with \eqref{eq:LLT}
\begin{gather}
   \label{eq:LTT}
   L_{1,j}^*T_{j} L_{1,j} = T_{j-1}, \quad
   L_{2,j}^*T_{j} L_{2,j} = T_{j-1}.
\end{gather}
If we multiply the first equality in \eqref{eq:LLT} by $L_{1,j}$ from the right and the second equality from the left by $L_{1,j}$, we obtain
\begin{align}
   \label{eq:LTLT}
   T_{j}L_{1,j} = L_{1,j}L_{2,j}^*L_{1,j} = L_{1,j}T_{j-1}.
\end{align}
Similarly, using the transposed equations in \eqref{eq:LLT} we obtain
\begin{align}
   \label{eq:LTLTtransposed}
   T_{j}^*L_{2,j} = L_{2,j}L_{1,j}^*L_{2,j} = L_{2,j}T_{j-1}^*,
\end{align}
hence also, for every $z\in\C$,
\begin{align}
   \label{eqV91}
   &R_{j}(z)L_{1,j} = L_{1,j}R_{j-1}(z), \\
   \label{eqV9}
   &R_{j}(z)^*L_{2,j} = L_{2,j}R_{j-1}(z)^*.
\end{align}


\subsection{Vectors and matrices involving the moments} 

Let $(s_k)_{k=0}^m$ be a Hausdorff positive definite sequence, see Definition~\ref{deboth}.
Hence, by definition, $H_{1,n}$ and $H_{2,n}$ are positive definite if $m=2n$ and
$H_{3,n}$ and $H_{4,n}$ are positive definite if $m=2n+1$.
Then the matrices $H_{1,n-1}$, $H_{2,n-1}$ respectively $H_{3,n-1}$, $H_{3,n-1}$ are also strictly positive, in particular they are invertible.
If we set
\begin{align}
   \label{27}
   Y_{r,j}:= \begin{pmatrix} s_j^{(r)} \\ \vdots \\ s_{2j-1}^{(r)}
   \end{pmatrix},
   \qquad 1 \le j\le n\ \text{if } r=1,3,4\ \text{and } 1 \le j \le n-1\ \text{if } r=2,
\end{align}
we can write $H_{r,j}$ as block matrix and obtain the following factorization
\begin{align}
   \label{eq:SchurFacH}
   \renewcommand*{\arraystretch}{1.5}
   H_{r,j} =
   \begin{pmatrix}[c:c]
      H_{r,j-1} & Y_{r,j} \\ \hdashline
      Y_{r,j}^* & s^{(r)}_{2j}
   \end{pmatrix}
   =
   \begin{pmatrix}[c:c]
      I & 0 \\ \hdashline
      Y_{r,j}^*H_{r,j-1}^{-1} & I
   \end{pmatrix}
   \begin{pmatrix}[c:c]
      H_{r,j-1} & 0 \\ \hdashline
      0         & \widehat H_{r,j}
   \end{pmatrix}
   \begin{pmatrix}[c:c]
      I & H_{r,j-1}^{-1} Y_{r,j} \\ \hdashline
      0 & I
   \end{pmatrix}
\end{align}
where
\begin{align}
   \label{lkkA}
   \widehat H_{r,j} := s^{(r)}_{2j} - Y_{r,j}^* H_{r,j-1}^{-1}Y_{r,j}
\end{align}
is the so-called \define{Schur complement of the block $s^{(r)}_{2j}$}.
It has been used in \cite[Equalities 2.34 and 2.35]{dette2} for $a = 0$ and $b = 1$.

\begin{remark}
   Recall that $H_{r,j}$ is strictly positive.
   Since the matrices on the left and right in the above factorization are adjoint to each other, $\widehat H_{r,j}$ is strictly positive too.
\end{remark}

Solving for the diagonal matrix in \eqref{eq:SchurFacH} gives
\begin{align}
   \label{eq:SchurFacHH}
   \renewcommand*{\arraystretch}{1.5}
   \begin{pmatrix}[c:c]
      H_{r,j-1} & 0 \\ \hdashline
      0         & \widehat H_{r,j}
   \end{pmatrix}
   =
   \begin{pmatrix}[c:c]
      I & 0 \\ \hdashline
      -Y_{r,j}^*H_{r,j-1}^{-1} & I
   \end{pmatrix}
   H_{r,j}
   \begin{pmatrix}[c:c]
      I & -H_{r,j-1}^{-1} Y_{r,j} \\ \hdashline
      0 & I
   \end{pmatrix}.
\end{align}

For $r=1,2,3,4$ we define the $(j+1)q\times q$ matrices
\begin{alignat}{3}
   \Sigma_{1,j} &:=
   \begin{pmatrix}
      -H_{1,j-1}^{-1}Y_{1,j} \\
      I
   \end{pmatrix},
   \qquad
   & \Sigma_{2,j} &:=
   \begin{pmatrix}
      -H_{2,j-1}^{-1}Y_{2,j} \\
      I
   \end{pmatrix},
   \label{sHHj0}
   \\[1ex]
   \Sigma_{3,j} &:=
   \begin{pmatrix}
      -H_{3,j-1}^{-1}Y_{3,j} \\
      I
   \end{pmatrix},
   \qquad
   & \Sigma_{4,j} &:=
   \begin{pmatrix}
      -H_{4,j-1}^{-1}Y_{4,j} \\
      I
   \end{pmatrix}.
   \label{sHHj1}
\end{alignat}

\begin{corollary}
   We have that
   $\widehat H_{r,j} = (Y_{r,j}^*,\  s_{2j}^{(r)})\Sigma_{r,j}$ and 
   \begin{align}
      \label{eq79}
      T_j H_{r,j}\Sigma_{r,j}=0
      \qquad\text{and}\qquad
      L_{2,j}^* H_{r,j}\Sigma_{r,j}=0.
   \end{align}
\end{corollary}
\begin{proof}
   The first formula follows directly from \eqref{lkkA} and the definition of the $\Sigma_{r,j}$.
   The factorization \eqref{eq:SchurFacHH} shows that
\begin{align}
   \renewcommand*{\arraystretch}{1.5}
   \begin{pmatrix}[c:c]
      I & 0 \\ \hdashline
      Y_{r,j}^*H_{r,j-1}^{-1} & I
   \end{pmatrix}
   \begin{pmatrix}[c:c]
      H_{r,j-1} & 0 \\ \hdashline
      0         & \widehat H_{r,j}
   \end{pmatrix}
   =
   H_{r,j}
   \begin{pmatrix}[c:c]
      \begin{matrix}
	 I \\ \hdashline 0
      \end{matrix}
      & \Sigma_{rj}
   \end{pmatrix},
\end{align}
hence, comparing the last column on both sides, we find that
\begin{align}
   \label{eq:lastcolumn}
   \renewcommand*{\arraystretch}{1.5}
   \begin{pmatrix}
      0_{nq\times q} \\ \hdashline
      \widehat H_{r,j}
   \end{pmatrix}
   & =
   H_{r,j} \Sigma_{rj}
\end{align}
from which \eqref{eq79} is an immediate consequence.
\end{proof}

Now from the negative first column of $H_{1,j}$
\begin{equation}
   \label{69yy}
   u_{j}:= -\begin{pmatrix} s_0 \\ \vdots \\ s_{j}
   \end{pmatrix}
\end{equation}
we construct the block vectors
$u_{r,j}$ by
\begin{alignat}{4}
   \label{70uu}
   u_{1,j} &:= T_j u_j,\\
   \label{hatu2}
   \widehat u_{2,j} &:= -L_{1,j+1}^*(I - b T_{j+1})(I - a T_{j+1}) u_{j+1},\\
   \label{21A}
   u_{3,j} &:= -(I - b T_j) u_j,\\
   \label{22}
   u_{4,j} &:= (I - a T_j) u_j
\end{alignat}
for $0\le j \le n-1$ for $r=2$ and $0\le j \le n$ in the other cases.
Moreover, we set
\begin{equation}
   u_{2,0}:=-(a+b)s_0+s_1, \quad
   u_{2,j}:=\widehat{u}_{2,j}+zv_{j}s_{0},
   \qquad 1\leq j\leq n-1.
   \label{u2jA}
\end{equation}
Note that if we delete the first block entry in each of these vectors, we obtain the negative first column of the block matrix $H_{r,j-1}$.
Therefore it is clear that 
\begin{align*}
   H_{r,n} - T_n\widetilde H_{r,n} = v_n u_{r, n+1}^* L_{2n}.
\end{align*}
It follows directly from the definition of $u_{1,j}$ and $u_j$ that
\begin{align}
   u_{n} = L_{1,n+1}^* u_{1,n+1},
   \qquad
   u_{1,n+1} = L_{1,n+1}u_{n}. \label{eqV11}
\end{align}

Moreover, note that 
\begin{align}
   \label{uuu001A}
   \widehat{u}_{2,j}=
   \begin{pmatrix}
      u_{2,0}\\
      -s^{(2)}_0 \\
      \vdots \\
      -s^{(2)}_{j-1} \\
   \end{pmatrix}
   \begin{aligned}[t]
      &= -(a+b)
      \begin{pmatrix} s_0 \\ \vdots \\ s_{j} \\
      \end{pmatrix}
      + ab
      \begin{pmatrix} 0 \\ s_0 \\ \vdots \\ s_{j-1} \\
      \end{pmatrix}
      +
      \begin{pmatrix} s_1 \\ \vdots \\ s_{j+1} \\
      \end{pmatrix}
      = b u_{4,j} + 
      \begin{pmatrix} s_{0}^{(4)} \\ \vdots \\ s_{j}^{(4)} \\
      \end{pmatrix}
      \\
      &= (a+b) u_{j} - ab T_n u_{j} + L_{1,j+1}^*u_{j+1}
      ,
      \quad 1\leq j\leq n-1.
   \end{aligned}
\end{align}


\subsection{Orthogonal matrix polynomials on ${\displaystyle[a,b]}$} 
Let ${\mathcal P}$ be the set of matrix polynomials
$P(t)=C_nt^n+\ldots+C_0$ with $q\times q$ matrix coefficients $C_n,\ldots, C_0$.
We denote by ${\rm deg}\,P:=\sup\{j\in \N\cup \{0\}: C_j\neq 0\}$ the
degree of $P$.
Note that the polynomial $P$ can be written as
\begin{equation}
   \label{pkkj}
   P(z) = \sum_{j=0}^n z^j a_j
   = (a_0, a_1, \dots, a_n)
   \begin{pmatrix}
      I \\ z \\ \vdots \\ z^n
   \end{pmatrix}
   = (a_0, a_1, \dots, a_n) R_n(z)v_n.
\end{equation}

Let $\sigma$ be a $q\times q$  positive measure on $[a,b]$.
We define a matrix inner product on the space ${\mathcal P}$:
\begin{equation}
 \langle P,Q \rangle_\sigma:=\int\limits_{[a,b]}P(t)\sigma(dt)Q^*(t). \label{sigma01}
\end{equation}
For details on matrix valued positive measures and the matrix inner product we refer to \cite{simon}.

Let us define the moments $s_j^{[\sigma]} := \int_{[a,b]} t^j\, \sigma(dt)$
and the matrix of moments 
$H_n^{[\sigma]} := (s_{j+k}^{[\sigma]})_{j,k=0}^n$.
Then for all matrix polynomials $P(z) = \sum_{j=0}^n z^j a_j$ and $Q(z) = \sum_{j=0}^n z^j b_j$ we have that
\begin{equation}
   \label{eq:PQ}
   \langle P,Q\rangle_\sigma = (a_0, a_1, \dots, a_n)
   \begin{pmatrix}
      s_0 & s_1 & \dots & s_n \\
      s_1 & s_2 & \dots & s_{n+1} \\
      \vdots &  & \dots & \vdots \\
      s_n & s_{n+1} & \dots & s_{2n}
   \end{pmatrix}
   \begin{pmatrix}
      b_0^{*} \\
      b_1^{*} \\
      \vdots \\
      b_n^{*}
   \end{pmatrix}
   = (a_0, a_1, \dots, a_n)
   H_{n}
   \begin{pmatrix}
      b_0^{*} \\
      b_1^{*} \\
      \vdots \\
      b_n^{*}
   \end{pmatrix}.
\end{equation}

The sequence $(P_j)_{j=0}^n$ is called a \define{finite sequence of orthogonal matrix polynomials with respect to $\sigma$}
if ${\rm deg}\, P_k=k$ and
\begin{align*}
   \langle P_j,P_\ell\rangle_\sigma=0, \quad  j\neq \ell.
\end{align*}

\begin{lem}
   If the matrix of moments $H_{j}$ is strictly positive, then there exists exactly one finite sequence of orthogonal matrix polynomials $(P_{k})_{k=0}^j$ with leading coefficient $I$.
   They are $P_{k}(z) = (-Y_{k}^*,\ I) R_k (z)v_k$.
\end{lem}
\begin{proof}
   If the matrix of moments $H_{j}$ is strictly positive, then so are all smaller matrices $H_{j'}$ for $0\le j' \le j$.
   We construct the orthogonal polynomials $P_k$ inductively.
   Clearly, $P_0 = I$.
   Now assume that we already have the polynomials $P_0,\,\dots,\, P_{\ell-1}$.
   If $P_{\ell}$ is orthogonal to their span, it is also orthogonal to the monomials $I, z,\, \dots,\, z^{\ell-1}$.
   Let $P_\ell(z) = z^\ell + \sum_{j=0}^{\ell-1} z^j a_j$.
   Then, using \eqref{eq:PQ}, the equations
   $\langle P_\ell,I\rangle_\sigma = \dots = \langle P_\ell, z^{\ell-1}\rangle_\sigma = 0$ can be written as
   \begin{align*}
      0 &=
      (a_0, \dots, a_{\ell-1}, I)
   \begin{tikzpicture}[baseline=(current bounding box.center),
      every left delimiter/.style={xshift=.5em},
      every right delimiter/.style={xshift=-.5em},
      ampersand replacement=\&,
      ]
      \matrix (m) [matrix of math nodes,nodes in empty cells, right delimiter={)},left delimiter={(} ]{
      s_0 \& s_1 \&  \hspace*{4ex} \& s_{\ell}  \\
      s_1 \& s_2 \& \& s_{\ell+1}  \\
      \&       \&  \\[2ex] 
      s_\ell \& s_{\ell+1} \& \& s_{2\ell}  \\
      } ;
%
      \draw[thick, dotted] ($(m-2-4.south west)+(0ex,0ex)$)-- ($(m-4-2.north east)+(0ex,-0ex)$);
      \draw[thick, dotted] ($(m-1-2.east)+(0ex,0ex)$)-- ($(m-1-4.west)+(0ex,-0ex)$);
      \draw[thick, dotted] ($(m-2-2.east)+(0ex,0ex)$)-- ($(m-2-4.west)+(0ex,-0ex)$);
      \draw[thick, dotted] ($(m-4-2.east)+(0ex,0ex)$)-- ($(m-4-4.west)+(0ex,-0ex)$);
      \draw[thick, dotted] ($(m-2-1.south)+(0ex,0ex)$)-- ($(m-4-1.north)+(0ex,-0ex)$);
      \draw[thick, dotted] ($(m-2-2.south)+(0ex,0ex)$)-- ($(m-4-2.north)+(0ex,-0ex)$);
      \draw[thick, dotted] ($(m-2-4.south)+(0ex,0ex)$)-- ($(m-4-4.north)+(0ex,-0ex)$);
   \end{tikzpicture}
   \begin{tikzpicture}[baseline=(current bounding box.center),
      every left delimiter/.style={xshift=.5em},
      every right delimiter/.style={xshift=-.5em},
      ampersand replacement=\&,
      ]
      \matrix (m) [matrix of math nodes,nodes in empty cells, right delimiter={)},left delimiter={(} ]{
      I \&  \hspace*{4ex} \& 0  \\ [3ex]
      0 \&       \& I  \\
      0 \&       \& 0  \\ 
      } ;
%
      \draw[thick, dotted] ($(m-1-1.south east)+(0ex,0ex)$)-- ($(m-2-3.north west)+(0ex,-0ex)$);
      \draw[thick, dotted] ($(m-1-1.east)+(0ex,0ex)$)-- ($(m-1-3.west)+(0ex,-0ex)$);
      \draw[thick, dotted] ($(m-2-1.east)+(0ex,0ex)$)-- ($(m-2-3.west)+(0ex,-0ex)$);
      \draw[thick, dotted] ($(m-3-1.east)+(0ex,0ex)$)-- ($(m-3-3.west)+(0ex,-0ex)$);
      \draw[thick, dotted] ($(m-1-1.south)+(0ex,0ex)$)-- ($(m-2-1.north)+(0ex,-0ex)$);
      \draw[thick, dotted] ($(m-1-3.south)+(0ex,0ex)$)-- ($(m-2-3.north)+(0ex,-0ex)$);
   \end{tikzpicture}
      =
      (a_0, \dots, a_{\ell-1}, I)
      \begin{pmatrix}
	 H_{\ell-1} \\[1ex] Y_{\ell}^*
      \end{pmatrix}
      \\
      &=
      (a_0, \dots, a_{\ell-1})
      H_{\ell-1}
      +
      Y_{\ell}^*
   \end{align*}
   which implies
   \begin{equation*}
      (a_0, \dots, a_{\ell-1})
      =
      - Y_{\ell}^* H_{\ell-1}^{-1} ,
      \qquad\text{hence}\qquad
      (a_0, \dots, a_{\ell-1}, I)
      =
      (-Y_{\ell}^* H_{\ell-1}^{-1}, I)
   \end{equation*}
   and therefore $P_\ell(z) = (-Y_{\ell}^* H_{\ell-1}^{-1},\ I) R_\ell(z)v_\ell$.
\end{proof}

Now we recall definitions of the OMP on $[a,b]$.
Note that for $r=1,2,$ the polynomials $P_{r,j}$ and $Q_{r,j}$
(resp. $P_{r+2,j}$ and $Q_{r+2,j}$)
were first introduced in \cite[page 936]{abpol} (resp. \cite[page 87]{thi}).
\begin{definition} \label{de002AAA}
   For $k=1,2$, let $(s_j)_{j=0}^{2n+k-1}$ be a Hausdorff positive definite sequence on $[a,b]$. 
   Furthermore, let $H_{r,j}$, $u_{r,j}$,  $Y_{r,j}$
   for $r=1,2,3,4$, $R_j$ and $v_j$ be as in \eqref{201A}, \eqref{70uu},
   \eqref{uuu001A}, \eqref{u2jA}, \eqref{21A}, \eqref{22}, \eqref{27},
   \eqref{52} and \eqref{59}, respectively.
   For $r=1,2,3,4$, we define the \define{polynomials of the first kind}
   \begin{alignat}{3}
      \label{Ppr} P_{r,0}(z)&:= I, & \quad P_{r,j}(z)  &:=\Sigma_{r,j}^*R_{j}(z)v_{j}.
   \end{alignat}
Furthermore, let
\begin{alignat}{3}
   \label{Qq1} Q_{1,0}(z)    &:=0, & \quad Q_{1,j}(z) &:=-\Sigma_{1,j}^*R_{j}(z)u_{1,j}, \\
   \label{Qq2} Q_{2,0}(z)&:=-(u_{2,0}+z\,s_0),& \quad
   Q_{2,j}(z) & := -\Sigma_{2,j}^*(z) R_{j}(z)u_{2,j}, \\
   \label{Qq3} Q_{3,0}(z)  &:=s_0,  & \quad Q_{3,j}(z) &:=\Sigma_{3,j}^* R_{j}(z)u_{3,j}, \\
   \label{Qq4} Q_{4,0}(z)  &:=-s_0, & \quad Q_{4,j}(z) &:=\Sigma_{4,j}^* R_{j}(z)u_{4,j}.
\end{alignat}
The matrix polynomials $Q_{r,j}$ are called \define{polynomials of the second kind}.
\end{definition}

Note that $(Q_{r,j})_{j=0}^n$ is not an orthogonal sequence with respect to $\sigma$.
For more details we refer the reader to
\cite[Sec. VII.6]{berezanskii}, \cite[Remark 4.2]{abpol} and \cite[Theorem 2.12]{thi}.
Note that in \cite{abpol,abRelations}
the polynomials $P_{r,j}$ (resp. $Q_{r,j}$) are denoted by
$\Gamma_{r,j}$ (resp. $\Theta_{r,j}$)
for $r=3,4$.

\begin{remark}\label{remAA1}
   Observe that $\Sigma_{3,j}$ depends on $b$ while $\Sigma_{4,j}$ depends on $a$.
   If we write $\Sigma_{3,j}(b)$ and $\Sigma_{4,j}(a)$ to make this dependence explicit, we easily see that 
   $\Sigma_{3,j}(a) = -\Sigma_{4,j}(a)$.
   Hence also the polynomials $P_{3,j}$ and $Q_{3,j}$ respectively
   $P_{4,j}$ and $Q_{4,j}$ depend on $b$ respectively $a$.
   If we denote them by $P_{3,j}(b,z)$, $Q_{3,j}(b,z)$
   and
   $P_{4,j}(a, z)$, $Q_{4,j}(a, z)$ we see that 
   \begin{align}
      & P_{3,n}(b,z)=P_{4,n}(b,z),\qquad  \quad P_{3,n}(a,z)=P_{4,n}(a,z),\label{eqQQPP1}\\
      & Q_{3,n}(b,z)=-Q_{4,n}(b,z),\qquad Q_{3,n}(a,z)=-Q_{4,n}(a,z).\label{eqQQPP2}
   \end{align}
\end{remark}

In the following we will suppress the parameters $a$ and $b$ in the matrices $\Sigma_{r,j}$ and the polynomials $P_{r,j}$ and $Q_{r,j}$.

\begin{remark}
   \label{rem:Q0P0}
   For $r=1,2,3,4$ and $z\in\C$ we have that 
   $Q_{r,n}(z) P_{r,n}(\bar z)^* = P_{r,n}(z)Q_{r,n}(\bar z)^*$.
\end{remark}
\begin{proof}
   In 
   \cite[Remark 4.7]{abKN} it is shown that there are solutions $\sigma_r$ to the moment problem such that 
   $s_r(z) := \int_{[a,b]} \frac{1}{t-z}\, \sigma_r(dt) = \frac{1}{p_r(z)} Q_{r,n}^*(\bar z) P_{r,n}^{*-1}(\bar z)$ 
   for $z\in\C\setminus [a,b]$
   and
   $p_1(z) = I$,
   $p_2(z) = (z-a)(b-z)$,
   $p_3(z) = (b-z)$,
   $p_4(z) = (z-a)$.
   Clearly, $s_r(z) = s_r(\bar z)^*$, hence we obtain
   $\frac{1}{p_r(z)} Q_{r,n}^*(\bar z) P_{r,n}^{*-1}(\bar z)
   = \frac{1}{p_r(z)} P_{r,n}^{-1}(z) Q_{r,n}(z)$
   and therefore 
   $P_{r,n}(z) Q_{r,n}^*(\bar z)
   = Q_{r,n}(z) P_{r,n}^*(\bar z)$.
   Since both products are polynomials, the equality holds for all $z\in\C$.
\end{proof}

In the following remark we give explicit relations between the Schur complements
$\widehat H_{k,j}$
and the polynomials $P_{1,j}$, $Q_{2,j}$,
$P_{3,j}$, $Q_{4,j}$
considered in Corollary 3.4 and Corollary 3.10 in \cite{abKN}.

\begin{remark} Let $\widehat H_{r,j}$,
   for
   $r=1,2,3,4$ be as in \eqref{lkkA}. Furthermore, let $P_{1,j}$, $Q_{2,j}$, $P_{3,j}$ and $Q_{4,j}$
   be as in Definition \ref{de002AAA}.
   The
   following equalities then hold:
   \begin{align}
      \widehat H_{1,j}=&-P_{1,j}(a)Q_{4,j}^*(a), \qquad
      \widehat H_{2,j-1}=-Q_{2,j-1}(a)P_{3,j}^{*}(a), \label{pqHHa}\\
      \widehat H_{3,j}=&P_{3,j}(a)Q_{2,j}^*(a), \qquad \quad
      \widehat H_{4,j}=Q_{4,j}(a) P_{1,j+1}^{*}(a).\label{pqHHa11}
   \end{align}
\end{remark}
Since the positive definiteness of the block Hankel matrix $H_{r,n}$ implies that also $\widehat H_{r,n}$ is positive definite, \eqref{pqHHa} and \eqref{pqHHa11} show the following fact.

\begin{remark}\label{remPaQa}
   For $k=1,2$, let $(s_j)_{j=0}^{2n+k-1}$ be a Hausdorff positive
   definite sequence on $[a,b]$. 
   Then the matrices
   $P_{1,j}(a)$, $Q_{4,j}(a)$, $Q_{2,j}(a)$  and $P_{3,j}(a)$
   are invertible.
\end{remark}

\section{Coupling identities}\label{sec0003}

We will use a number of identities to describe the explicit relation between the two resolvent matrices
$V^{(2n+1)}$ from \eqref{eq17}
and 
$U^{(2n+1)}$ from \eqref{RM2nm1} if $m$ is even
and between
$V^{(2n)}$ from \eqref{eqn:VV14}
and 
$U^{(2n)}$ from \eqref{RM1} if $m$ is odd.
\medskip

For $r=1,2,3,4$, the following Ljapunov type identities
\begin{align}
  &H_{r,j}T_j^*-T_j H_{r,j}=u_{r,j}v_j^*-v_j u_{r,j}^* 
  \label{eq78}
\end{align}
are called the \define{fundamental identities of the THMM problem}.
They are crucial for proving that the associated solution $s(z)$ to  a given THMM problem is a solution of a system of two matrix inequalities \cite{D-C}, \cite{abdon1}, \cite{abdon2}.

Recall the definition of $H_{r,j}$ and $\widetilde H_{r,j}$ from \eqref{201A} and \eqref{eqHt1n}.
Then we have that
\begin{equation*}
   \label{eq:HH}
   \widetilde H_{r,j}
   = L_{2,j+1}^* H_{r,j+1} L_{1,j+1}
   = L_{1,j+1}^* H_{r,j+1} L_{2,j+1}
\end{equation*}
and
\begin{align}
   H_{1,j}L_{1,j+1}^*-L_{2,j+1}^*H_{1,j+1}T_{j+1}^*=0,\label{eq83a}
\end{align}

\begin{remark}\label{rem3.1}
   Let $H_{r,j}$, $\widetilde H_{1,j}$, $T_j$, $L_{2,j}$, $L_{1,j}$,
   $v_j$, $u_{j}$, $u_{1,j}$, $u_{3,j}$, $u_{4,j}$,
   be as in 
   \eqref{201A},   
   \eqref{eqHt1n}, 
   \eqref{eqTn}, 
   \eqref{82g},  
   \eqref{59},   
   \eqref{69yy}, 
   \eqref{70uu}, 
   \eqref{22},   
   respectively.
  Let us give some relations between the matrices of moments.
  Clearly
  \begin{subequations}
  \begin{align}
     \label{eq833}
     v_j u_j^*-T_j\widetilde H_{1,j}+H_{1,j}=0, \\
     \label{eq833transposed}
     u_j v_j^* - \widetilde H_{1,j}T_j^* +H_{1,j}=0.
  \end{align}
  \end{subequations}
   Applying $L_{2,j}^*$ from the left to \eqref{eq833transposed} and using that $L_{2,j}^*\widetilde H_{1,j} T_j^*  = \widetilde H_{1,j-1}L_{1,j}^*$ we obtain
   \begin{align}
      \label{eqV4}
      u_{j-1} v_j^* - \widetilde H_{1,j-1}L_{1,j}^* +L_{2,j}^*H_{1,j} = 0.
   \end{align}
   The following relations between the $H_{r,j}$ are easy to see:
  \begin{align}
     \label{eq:H3H4}
     H_{3,j} = b H_{1,j} - \widetilde H_{1,j} \qquad\text{and}\qquad
     H_{4,j} = -a H_{1,j} + \widetilde H_{1,j},
  \end{align}
  and
  \begin{align}
     \nonumber
     H_{2,j-1} &= -ab H_{1,j-1} + (a+b) \widetilde H_{1,j-1} - L_{2,j}^*\widetilde H_{1,j}L_{1,j} \\
     \label{eq:H2H3H4}
     & = b H_{4,j-1}- \widetilde H_{4,j-1}  = -a H_{3,j-1}+ \widetilde H_{3,j-1}.
  \end{align}
  Multiplication of \eqref{eq833transposed} by $a$ respectively $b$ together with \eqref{eq:H3H4} yields
  \begin{align}
     &a u_j v_j^*+\widetilde H_{1,j}(I-a T_j^*)-H_{4,j}=0, \label{eqV670}
     \\
     &bu_j v_j^*+\widetilde H_{1,j}(I-b T_j^*)+H_{3,j}=0. \label{eqV672}
  \end{align}
  Moreover, \eqref{eq:H3H4} implies that 
  \begin{align}
     \label{eq:H4H1H3}
     H_{4,j}^{-1}\widetilde H_{1,j}H_{3,j}^{-1}=H_{3,j}^{-1}\widetilde H_{1,j}H_{4,j}^{-1} 
  \end{align}
  as follows from
  \begin{align*}
     H_{4,j}\widetilde H_{1,j}^{-1}H_{3,j}
     & =  \big[ \widetilde H_{1,j} - a H_{1,j} \big] \widetilde H_{1,j}^{-1}
     \big[ - \widetilde H_{1,j} + b H_{1,j} \big] 
     \\
     & =  \big[ -\widetilde H_{1,j} + b H_{1,j} \big] \widetilde H_{1,j}^{-1}H_{3,j} 
     \big[ \widetilde H_{1,j} - a H_{1,j} \big] 
     = H_{3,j}\widetilde H_{1,j}^{-1}H_{4,j}.
  \end{align*}
  Inserting the expressions for $\widetilde H_{1,n}$ from \eqref{eq:H3H4} in \eqref{eq833} and \eqref{eq833transposed}, we obtain
  \begin{subequations}
     \begin{align}
	T_jH_{3,j} &= -v_j u_j^* - (I-bT_j)H_{1,j}, \label{eq83}\\
	H_{3,j}T_j^* &= - u_j v_j^* - H_{1,j}(I-bT_j^*), \label{eqV66} \\
	\stepcounter{parentequation}
	\gdef\theparentequation{\arabic{parentequation}}
	\setcounter{equation}{0}
	T_jH_{4,j} &= v_j u_j^* + (I-aT_j)H_{1,j}, \label{eq82}\\
	H_{4,j}T_j^* &= u_j v_j^* + H_{1,j}(I-aT_j^*). \label{eqV66a}
     \end{align}
  \end{subequations}
  An analogous equation for $H_{2,j}$ is
   \begin{align}
      \label{eqV77}
      T_jH_{2,j}=
      -bu_{4,j}v_{j}^* + v_j \widehat u_{2,j}^* - H_{4,j}(I-bT_j^*)
   \end{align}
   because by \eqref{eq:H2H3H4} and \eqref{eq78} for $r=4$ we have that
   \begin{align*}
      T_jH_{2,j}
      & = b T_j H_{4,j} - T_j \widetilde H_{4,j}
      = b H_{4,j}T_j - b u_{4,j}v_j^* + b v_j u_{4,j}^* 
      - H_{4,j} + v_j (s_0^{(4)}, \dots, s_j^{(4)})
      \\
      & = - b u_{4,j}v_j^* - H_{4,j}(I - b T_j )
      + v_j \big( b u_{4,j}^* + (s_0^{(4)}, \dots, s_j^{(4)}) \big)
      \\
      & = - b u_{4,j}v_j^* - H_{4,j}(I - b T_j )
      + v_j \widehat u_{2,j}^* .
   \end{align*}
   In the last step we used \eqref{uuu001A}.
   \bigskip
\end{remark}

\begin{remark}\label{remA2}
   Let 
   $H_{r,j}$, $\widetilde H_{1,j}$, $T_j$, $L_{2,j+1}$, $L_{1,j}$,
   $v_j$,  $u_{1,j}$, $\widehat u_{2,j}$ and $u_{4,j}$,
   be as in
   \eqref{201A},   
   \eqref{eqHt1n}, 
   \eqref{eqTn}, 
   \eqref{82g},  
   \eqref{59},   
   \eqref{70uu}, 
   \eqref{uuu001A} 
   and \eqref{22},   
   respectively.
   The next two equalities are obtained from \eqref{eqV66} and \eqref{eqV66a} by multiplication from the left by ${}-(I-bT_j)$ and $(I-aT_j)$ respectively.
   \begin{align}
      u_{3,j} v_j^* - (I-bT_j) \big[ H_{3,j}T_j^* + H_{1,j}(I-bT_j^*) \big]
      &=0, \label{eqV67}
      \\
      u_{4,j} v_j^* - (I-aT_j) \big[ H_{4,j}T_j^*-H_{1,j}(I-aT_j^*) \big]
      &=0. \label{eqV671}
   \end{align}
   Equations \eqref{eqV66} and \eqref{eqV66a} together with the fundamental identity \eqref{eq78} for $r=1$ and 
   $-u_jv_j^* + b u_{1,j}v_j^* = u_{3,j}v_j^*$
   and
   $-u_jv_j^* - a u_{1,j}v_j^* = u_{4,j}v_j^*$
   yield
   \begin{align}
      H_{3,j}T_j^* &= u_{3,j}v_{j}^*- b v_{j}u_{1,j}^* -(I-bT_j)H_{1,j},
      \label{eqV88}\\
      H_{4,j}T_j^* &= u_{4,j}v_{j}^* + a v_{j} u_{1,j}^* + (I-aT_j)H_{1,j}.
      \label{eqV99}
   \end{align}
   Taking the adjoint of \eqref{eqV670} and multiplying from the right by $L_{1,j+1}^*$ gives
   \begin{align}
      av_{j}u_{1,j+1}^*-H_{4,j} L_{1,j+1}^*+(I-a T_j)\widetilde H_{1,j}L_{1,j+1}^*=0. \label{eqV41}
   \end{align}
   Using that $\widetilde H_{1,j}L_{1,j}^* = L_{2,j+1}^* H_{1,j+1} - u_j v_{j+1}^*$ we obtain from \eqref{eqV41}
   \begin{align}
      \label{eqV7}
      u_{4,j}v_{j+1}^* + av_j u_{1,j+1}^* -H_{4,j}L_{1,j+1}^*+(I-aT_j)L_{2,j+1}^*H_{1,j+1}=0.
   \end{align}
\end{remark}

The identities of the next lemma are used to prove our main results Theorem~\ref{thmain} and Theorem~\ref{thmainodd}.

\begin{lem}\label{lemA1}
   Let  $T_j$, $L_{2,j}$, $L_{1,j}$, $v_j$,   $u_{r,j}$ for $r=1,3,4$, $\widehat u_{2,j}$,  %
   $\Sigma_{r,j}$ for $r=1,3,4$, be as in
   \eqref{eqTn}, \eqref{82g}, \eqref{59},
   \eqref{70uu},\eqref{22}, \eqref{uuu001A}, \eqref{sHHj0} and \eqref{sHHj1}, respectively.
   Then the next equalities hold.
     \begin{align}
	\big[ v_ju_j^*+(I-b T_j)H_{1,j} \big] \Sigma_{3,j} &= 0, \label{pqH00} \\
	\big[ b u_{4,j}v_j^*-v_j \widehat u_{2,j}^*+H_{4,j}(I-b T_j^*)\big] \Sigma_{2,j} &= 0, \label{eq362a}\\
	\big[ v_j u_{3,j}^*-b u_{1,j}v_j^*-H_{1,j}(I-b T_j^*) \big] \Sigma_{3,j} &= 0, \label{eqn362}
	\\
	\big[ v_j u_{4,j}^* + a u_{1,j}v_j^*+H_{1,j}(I-a T_j^*) \big] \Sigma_{4,j} &= 0, \label{eqn361}
	\\
	\big[ u_{4,j}v_{j+1}^* +av_ju_{1,j+1}^* - H_{4,j}L_{1,j+1}^* \big] \Sigma_{1,j+1} &= 0,\label{eqV10}
	\\
	\big[ -H_{4,n}R_n^*(a)L_{1,n+1}^*+R_n(a)u_{4,n}v_{n+1}^*R_{n+1}^*(a) \big] \Sigma_{1,n+1} &= 0,
	\label{eq502}
	\\
	\Sigma_{4,j}^*
	\big[ bu_{4,j}v_j^{*}+a v_j u_{3,j}^*-(b-a)
	R_j(a)u_{4,j}v_j^* R_j^{*}(a) \big] \Sigma_{3,j} &= 0.\label{eqn363}
     \end{align}
\end{lem}

\begin{proof}
   Recall that 
   $L_{2,j+1}^* H_{1,j+1}\Sigma_{1,j+1}=0$ by \eqref{eq79}.
   The equalities \eqref{pqH00} and \eqref{eq362a} follow from \eqref{eq79} because the left hand sides are equal to 
   $-T_jH_{3j}\Sigma_{3j}$ by \eqref{eq83} and $-T_jH_{2, j}\Sigma_{2, j}$ by \eqref{eqV77} respectively.
   Equalities \eqref{eqn362} and \eqref{eqn361} are a consequence of \eqref{eq78} 
   and \eqref{pqH00} and \eqref{eq362a} respectively.
   Equality \eqref{eqV10} follows from \eqref{eqV7}.
   To prove Equality \eqref{eqn363}, note that by definition of $u_{3,j}$ and $u_{4,j}$ and by 
   \eqref{eqV66a} and \eqref{eq83}, we have
   \begin{multline*}
      bu_{4,j}v_j^{*}+a v_j u_{3,j}^*
      = bu_{j}v_j^{*} - a v_j u_{j}^* 
      +  ab (v_ju_{j}^{*} T_j^* -  T_ju_{j}v_j^{*})
      \\
      \begin{aligned}
      &= b H_{4,j} T_j^* - b H_{1,j}(I-aT_j^*) + aT_jH_{3,j} + a(I-bT_j)H_{1,j} 
      + ab(v_ju_{j}^{*} T_j^* -  T_ju_{j}v_j^{*} )
      \\
      &= b H_{4,j} T_j^* + aT_jH_{3,j} 
      - (b-a) H_{1,j}
      + ab(v_ju_{j}^{*} T_j^* -  T_ju_{j}v_j^{*} -T_jH_{1,j} + H_{1,j} T_j^*)
      \\
      &= b H_{4,j} T_j^* + aT_jH_{3,j} 
      - (b-a) H_{1,j}
      \end{aligned}
   \end{multline*}
   where in the last step we used \eqref{eq78}.
   Inserting this in \eqref{eqn363} and using \eqref{eq79} for $r=3$ and $r=4$, we obtain
   \begin{align*}
      \Sigma_{4,j}^*
      \big[ bu_{4,j}v_j^{*}+a v_j u_{3,j}^*-(b-a)
      R_j(a)u_{4,j}v_j^* R_j^{*}(a)\big] \Sigma_{3,j}
      \qquad
      \qquad
      \qquad
      \\
      \begin{aligned}
	 &=-(b-a)\Sigma_{4,j}^*
	 \big[ H_{1,j} + R_j(a)u_{4,j}v_j^* R_j^{*}(a) \big] \Sigma_{3,j}
	 \\
	 &=-(b-a)\Sigma_{4,j}^*
	 \big[ H_{1,j}(I - aT_j^*) + u_{j}v_j^* \big] R_j^{*}(a) \Sigma_{3,j}
	 \\
	 &=-(b-a)\Sigma_{4,j}^*H_{4,j}T_j^*R_j(a)\Sigma_{3,j}
	 =0.
      \end{aligned}
   \end{align*}
   In the second to last equality we employed \eqref{eq82}.

   Finally we prove Equality \eqref{eq502}. We have
   \begin{multline*}
      \qquad 
      \big[ -H_{4,n}R_n^*(a)L_{1,n+1}^*+R_n(a)u_{4,n}v_{n+1}^*R_{n+1}^*(a) \big] \Sigma_{1,n+1}
      \\
      \begin{aligned}
	 & = \big[ -H_{4,n}L_{1,n+1}^* + u_{n}v_{n+1}^* \big] R_{n+1}^*(a)\Sigma_{1,n+1}\\
	 &= \big[ -(-a H_{1,n}+\widetilde H_{1,n})L_{1,n+1}^*
	 + u_{n}v_{n+1}^* 
	 \big] R_{n+1}^*(a)\Sigma_{1,n+1}
	 \qquad 
	 \\
	 & = \big[ a L_{2,n+1}^* H_{1,n+1} T_{n+1}^* - L_{2,n+1}^*H_{1,n+1} \big] R_{n+1}^*(a)\Sigma_{1,n+1}\\
	 & =-L_{2,n+1}^*H_{1,n+1}\Sigma_{1,n+1}
	 =0.
      \end{aligned}
   \end{multline*}
   In the first equality we used \eqref{eqV91}. 
   The second equality follows from the second equality of \eqref{eq:H3H4} and the third equality follows from \eqref{eqV4}. 
   In the last step we used \eqref{eq:lastcolumn}.
\end{proof}

\begin{remark} 
   \label{rem660} 
   Let $\widehat H_{1,j}$ be as in \eqref{lkkA} and let $Q_{4,j}$ be as in Definition~\ref{de002AAA}.
   Then
   \begin{align}
      &v_j^*R_j^*(\bar z)H_{1,j}^{-1}R_j(b)v_jQ_{3,j}^*(b) = P_{3,j}^*(\bar z), \label{pqH0A}\\
      &\widehat H_{1,j} = P_{1,j}(b)Q_{3,j}^*(b),\label{pqH1A}\\
      &
      I-(z-a)u_{1,j}^*
      R_j^*(\bar z)H_{1,j}^{-1}R_j(a)v_j=Q_{4,j}^{*}(z)Q_{4,j}^{*-1}(a),
      \label{pqH2B}\\
      &I-(z-b)u_{1,j}^* R_j^* (\bar z)H_{1,j}^{-1}R_j(b)v_j=Q_{3,j}^{*}(z)Q_{3,j}^{*-1}(b).
      \label{pqH2C}
   \end{align}
\end{remark}
\begin{proof}
   First of all note that the right hand sides of \eqref{pqH2B} and \eqref{pqH2C} make sense because
   $Q_{4,j}(a)$ is invertible by Remark~\ref{remPaQa}
   and $Q_{3,j}(b)$ is invertible by \eqref{pqH1A}.
   \smallskip

   \noindent
   We prove \eqref{pqH0A} by using \eqref{pqH00} in the third step:
   \begin{align*}
      v_j^*R_j^*(\bar z)H_{1,j}^{-1}R_j(b)v_jQ_{3,j}^*(b)
      & = -v_j^* R_j^*(\bar z)H_{1,j}^{-1} R_j(b)v_ju_{3,j}^*R_j^*(b) \Sigma_{3,j}\\
      & = -v_j^* R_j^*(\bar z)H_{1,j}^{-1} R_j(b)v_ju_{j}^* \Sigma_{3,j}\\
      & = v_j^* R_j^*(\bar z)H_{1,j}^{-1} R_j(b) (I-bT_j) H_{1,j} \Sigma_{3,j}\\
      & = v_j^* R_j^*(\bar z) \Sigma_{3,j}
      = P_{3,j}^*(\bar z).
   \end{align*}
   To prove \eqref{pqH1A}, we calculate the inverse of the matrix $H_{1,j}$ using the factorization \eqref{eq:SchurFacH}, see also \cite[Equality (3.8)]{abpol}:
   \begin{equation}\label{eqA01}
      H_{1,j}^{-1}=
      \begin{pmatrix}
	 H_{1,j-1}^{-1} & 0 \\[1ex]
	 0 & 0
      \end{pmatrix}+
      \begin{pmatrix}
	 -H_{1,j-1}^{-1}Y_{1,j} \\[1ex]
	 I
      \end{pmatrix}\widehat H_{1,j}^{-1}(-Y_{1,j}^*H_{1,j-1}^{-1},\ I).
   \end{equation}
   Inserting \eqref{eqA01} in Equality~\eqref{pqH0A} and using the definition \ref{de002AAA} for $P_{1,j}$, we find that 
   \begin{align*}
      P_{3,j}^*(\bar z)
      &= 
      v_j^*R_j^*(\bar z)H_{1,j}^{-1}R_j(b)v_jQ_{3,j}^*(b) \\
      &= v_{j-1}^*R_{j-1}^*(\bar z)H_{1,j-1}^{-1}R_{j-1}(b)v_{j-1}Q_{3,j}^*(b)+P_{1,j}^*(\bar z)\widehat H_{1,j}^{-1}P_{1,j}(b)Q_{3,j}^*(b).
   \end{align*}
   The first term in the sum above is a polynomial of degree $j-1$ in $z$.
   Since the leading term in the polynomials $P_{1,j}^*$ and $P_{3,j}^*$ is $I$,
   equating  coefficients for $z^j$, shows that
   $$
   \widehat H_{1,j}^{-1}P_{1,j}(b)Q_{3,j}^*(b) = I
   $$
   which is equivalent to \eqref{pqH1A}.
   \\
   The Equality \eqref{pqH2B} is proved in Equality (3.6) of \cite[Lemma 3.2]{abKN} where the notation
   $\Theta_{2,j}$ is used instead of $Q_{4,j}$.
   \\
   To prove \eqref{pqH2C}, we verify the equivalent equality
   \begin{align*}
      \big[ I-(z-b)u_{1,j}^*R_j^*(\bar z)H_{1,j}^{-1}R_j(b)v_j \big] Q_{3,j}^{*}(b)-Q_{3,j}^{*}(z)=0.
   \end{align*}
   Using \eqref{eqQQPP1} and \eqref{eqQQPP2} and the Definition \ref{de002AAA} for $Q_{3,j}$, we obtain
   \begin{align}
      \nonumber
      \big[ I- & (z-b) u_{1,j}^*R_j^*(\bar z)H_{1,j}^{-1}R_j(b)v_j \big] Q_{3,j}^{*}(b)-Q_{3,j}^{*}(z)\\
      \label{eq:pqH2Cright}
      &=
      - \big[ u_{3,j}^*R_j^*(b) - (z-b)u_{1,j}^*R_j^*(\bar z)H_{1,j}^{-1}R_j(b)v_j u_{3,j}^*R_j^*(b)
      -u_{3,j}^*R_j^*(\bar z) \big] \Sigma_{3,j}.
   \end{align}
   Using that $u_{1,j}^* = u_j^* T_j^*$, $u_{3,j}^* R_j(b)^* = u_j^*$ and that
   \begin{align*}
      u_{3,j}^*R_j^*(b) -u_{3,j}^*R_j^*(\bar z)
      = u_{3,j}^* (b-z)R_j^*(b) T_j^* R_j^*(\bar z)
      = (b-z) u_{j}^* T_j^* R_j^*(\bar z),
   \end{align*}
   the right hand side of \eqref{eq:pqH2Cright} becomes
   \begin{multline*}
      (z-b) \big[ u_j^* T_j^* R_j^*(\bar z) + u_{j}^* T_j^* R_j^*(\bar z)H_{1,j}^{-1}R_j(b)v_j u_{j}^* \big] \Sigma_{3,j}
      \\
      =(z-b)u_j^*T_j^*R_j^*(\bar z)H_{1,j}^{-1}R_j(b)
      \big[ (I-bT_j)H_{1,j}+v_ju_j^* \big] \Sigma_{3,j}
      =0.
   \end{multline*}
   The last equality is obtained from \eqref{pqH00}.
\end{proof}
\begin{remark}\label{remPbQb}
   Let $(s_j)_{j=0}^{2n}$ be a Hausdorff positive definite sequence on $[a,b]$.
   Furthermore, let $P_{1,j}$  and $Q_{4,j}$ be as in Definition \ref{de002AAA}.
   Then the matrix $P_{1,j}(b)$ is invertible by \eqref{pqH1A}. 
\end{remark}
%

\section{Explicit relation between two resolvent matrices of the THMM Pro\-blem via OMP}
\label{sec0004} 
 In this section we give a representation of the resolvent matrix $U^{(m)}$ of the THMM problem in terms of orthogonal matrix polynomials
 both for an an odd and even number of gvien moments.
 Let
 \begin{align}\label{eqn:11}
    J_q:= \begin{pmatrix}
       0   & -iI  \\ iI    &  0
    \end{pmatrix}
 \end{align}
and
\begin{align}\label{eqn432}
   {\mathfrak J}_q:=\begin{pmatrix}
      0 & I\\ I & 0
   \end{pmatrix}
\end{align}
where the entries $I$ and $0$ are $q\times q$ matrices.
The matrices $J_q$ and ${\mathfrak J}_q$ satisfy 
\begin{align}
   J_q=&J_q^*, \qquad J_q^2 = I,  \label{eqJJ}\\
   {\mathfrak J}_q=&{\mathfrak J}_q^*, \qquad {\mathfrak J}_q^2 = I.
   \label{eqJJ1}
\end{align}
The properties of the matrix $J_q$ were used to explain relevant results concerning
the next four matrices:
\begin{definition} \label{def2001B}
   Let $(s_k)_{k=0}^{2n}$\  $((s_k)_{k=0}^{2n+1})$ be a Hausdorff positive definite sequence on $[a,b]$.
   Let $H_{r,n}$, $u_{r,n}$, for
   $r=1,2,3,4$, $R_n$, $v_n$ and $J_q$ be defined as in \eqref{201A},
   \eqref{70uu}, \eqref{uuu001A}, \eqref{22}, \eqref{52}, \eqref{59} and \eqref{eqn:11},
   respectively.
   For $r=1,2,3,4$, the \define{Kovalishina resolvent matrix of the THMM problem} is
   \begin{align}\label{RMKov}
      \widetilde V^{(n)}_r(z):=I-iz
      \begin{pmatrix} v_n^* \\ u_{r,n}^* 
      \end{pmatrix}
      R_n^*(\bar z)
      H_{r,n}^{-1}(v_n,\ u_{r,n})J_q.
   \end{align}
\end{definition}
The name Kovalishina resolvent matrix was suggested by Yu.~Dyukarev as
the work \cite{D-C} was being prepared.  Irina Kovalishina introduced and solved interpolation problems in the Nevanlinna class of functions as well as the
matrix moment problem on the real axis; see \cite{kov0}, \cite{kov1}.

Two relevant properties \cite{D-C} distinguish the matrix $\widetilde V^{(n)}_r$.
Firstly, we mention the inequality
\begin{equation}\label{eqVn1}
  J_q-\widetilde V^{(n)}_r(z)J_q\widetilde V^{(n)*}_r(z)\leq 0
\end{equation}
for $z$ to the upper half plane of $\C$.
Secondly,
 the inverse matrix of $\widetilde V^{(n)}_r$ can be expressed as follows:
\begin{equation}\label{eqVn2}
  \left[ \widetilde V^{(n)}_r(z) \right]^{-1} =J_q\widetilde V^{(n)*}_r(\bar z)J_q.
\end{equation}

The relations \eqref{eqVn1} and \eqref{eqVn2} were used to find the solutions the THMM problem as well as a factorization of the resolvent matrix; see \cite{abmult}, \cite{abODD}, \cite{abEVEN} and \cite{abdonDS1}.

The formulas in the next remark can be readily shown by direct calculations.
\begin{remark}\label{rem00A}
 The matrix defined in \eqref{RMKov} can be written in the following form
 \begin{equation}
    \widetilde V_r^{(n)}(z):=
    \begin{pmatrix}
       \widetilde \alpha_r^{(n)}(z) & \widetilde\beta_r^{(n)}(z)\\
       \widetilde\gamma_r^{(n)} (z)& \widetilde\delta_r^{(n)}(z)
    \end{pmatrix},
    \quad z\in {\mathbb C},
    \label{RM2nm1A}
 \end{equation}
 where
 \begin{align}
    \widetilde \alpha_r^{(n)}(z):=&I+z v_n^*R_n^*(\bar z)H_{r,n}^{-1}u_{r,n}, \label{eq11ar}\\
    \widetilde\beta_r^{(n)}(z):=&
    -z v_n^*R_n^*(\bar z)H_{r,n}^{-1}v_n, \label{eq12ar}\\
    \widetilde\gamma_r^{(n)}(z):=&
    z u_{r,n}^*R_n^*(\bar z)H_{r,n}^{-1}u_{r,n}, \label{eq21ar}\\
    \widetilde\delta_r^{(n)}(z):=&
    I-z u_{r,n}^*R_n^*(\bar z)H_{r,n}^{-1}v_n.\label{eq22ar}
 \end{align}
\end{remark}



\subsection{Case of an even number of moments} 
In this section, we consider the explicit relation between the resolvent matrix of the THMM problem
for the case of an even number of moments introduced in \cite{D-C} and the 
resolvent matrix presented in \cite{abKN} and in \cite{abdon1}.

\begin{assumption}
   \label{ass:H1tilde}
   In this section we assume that
   \begin{align*}
      \widetilde H_{1,n}\ \text{ and }\
      \widetilde H_{1,n-1}\quad \text{are invertible.}
   \end{align*}
\end{assumption}

\begin{remark}\label{remZ0A}
   Note that 
   \begin{equation}\label{eqHtH34}
     \widetilde H_{1,n} = (b-a)^{-1} ( aH_{3,n} + bH_{4,n}).
   \end{equation}
   Hence, if $0\le a < b$ or $a < b \le 0$, then $\widetilde H_{1,n}$ is strictly positive or negative, hence so is $\widetilde H_{1,n-1}$ and Assumption~\ref{ass:H1tilde} is satisfied, see \cite[Prop. 1]{AbBa}.
\end{remark}

\begin{definition} \label{def2001BB}
Let $(s_k)_{k=0}^{2n+1}$ be a Hausdorff positive definite sequence on $[a,b]$. 
Let $H_{3,n}$ and $H_{4,n}$, $\widetilde{H}_{1,n}$, $v_n$ and $u_n$
as in \eqref{201A}, \eqref{eqHt1n}, \eqref{59} and \eqref{69yy}, respectively, and define
\begin{align}
   M^{(2n+1)}&:=-a u_n^{*} \widetilde{H}_{1,n}^{-1}u_n,
   \qquad
   N^{(2n+1)}:=-bv_n^{*}H_{4,n}^{-1} \widetilde{H}_{1,n}H_{3,n}^{-1}v_n\label{MN4n}
\end{align}
with
\begin{align*}
   C^{(2n+1)}&:=
   \begin{pmatrix}
      I       & 0 \\
      M^{(2n+1)}  & I
   \end{pmatrix},
   \quad
   D^{(2n+1)}:=
   \begin{pmatrix}
      I & N^{(2n+1)} \\
      0 & I
   \end{pmatrix}
\end{align*}
and
\begin{align}\label{eqn:VV14m1}
   V_4^{(2n+1)}(z):=&\widetilde V_4(z)
   C^{(2n+1)}D^{(2n+1)}.
\end{align}
The matrix \eqref{eqn:VV14m1} is called the 
\define{first auxiliary resolvent matrix of the THMM problem in the case of an even number
of moments}.
In \cite{D-C} it is denoted by $U_r(z)$.
\smallskip

\noindent
Finally, the matrix
\begin{align}\label{eq17}
   V^{(2n+1)}(z):=
   \begin{pmatrix}
      I & 0             \\
      0 & (z-a)^{-1}I
   \end{pmatrix}
   V^{(2n+1)}_4(z)
   \begin{pmatrix}
      I & 0             \\
      0 & (z-a)I
   \end{pmatrix}
\end{align}
is called the 
\define{resolvent matrix of the THMM problem for the case of an even number of moments}.
This matrix is defined for $z\in{\mathbb C}\setminus\{a\}$.
\end{definition}
\noindent
Note that the point $z=a$ in \eqref{eq17} is a removable singularity; 
see \cite[Theorem 4]{D-C}. 

\begin{definition} \label{rmodd2n1} 
   Let $(s_j)_{j=0}^{2n+1}$ be 
   a Hausdorff positive definite sequence on $[a,b]$.
   Let $P_{1,n+1}$, $Q_{1,n+1}$, $P_{2,n}$, and
   $Q_{2,n}$ be as in Definition \ref{de002AAA}.
   The $2q\times 2q$ matrix polynomial
   \begin{equation}
      U^{(2n+1)}(z):=
      \begin{pmatrix}
	 \alpha^{(2n+1)}(z) & \beta^{(2n+1)}(z)\\
	 \gamma^{(2n+1)} (z)& \delta^{(2n+1)}(z)
      \end{pmatrix},
      \quad z\in {\mathbb C},
      \label{RM2nm1}
   \end{equation}
   with
   \begin{align}
      \alpha^{(2n+1)}(z) &:= Q_{2,n}^*(\bar z)Q_{2,n}^{*-1}(a), \label{eq11m1}\\
      \beta^{(2n+1)}(z) &:= -Q_{1,n+1}^*(\bar z)P_{1,n+1}^{*-1}(a), \label{eq12m1}\\
      \gamma^{(2n+1)}(z) &:= -(b-z)(z-a)P_{2,n}^*(\bar z)Q_{2,n}^{*-1}(a), \label{eq21m1}\\
      \delta^{(2n+1)}(z) &:= P_{1,n+1}^*(\bar z)P_{1,n+1}^{*-1}(a)\label{eq22m1}
   \end{align}
   is called the 
   \define{
   resolvent matrix of the THMM problem in the point $a$ in the case of an even number of moments.
   }
\end{definition}
\begin{remark}\label{remAA23}
   \begin{enumerate}[a)]

      \item By Assumption~\ref{ass:H1tilde}, 
      ${\widetilde H}_{1,n}$ is strictly positive, hence its 
      Schur complement ${\widehat {\widetilde H}}_{1,n}$ is well-defined and can be expressed with the help of the
      polynomials $P_{1,n+1}(0)$ and $\widetilde Q_{2,n}(0)$ \cite[Equality (4.54)]{abinfty}
      \begin{equation}\label{eqttH1n}
	 {\widehat {\widetilde H}}_{1,n}=-\widetilde Q_{2,n}(0)P_{1,n+1}^*(0),
      \end{equation}
      where 
      \begin{align*}
	 \widetilde Q_{2,n}(z)=-\big(-(s_{n+1},\, \dots,\, s_{2n})^* \widetilde H_{1,n-1}^{-1},\ I)R_n(z)u_n.
      \end{align*}
      Note that $\widetilde Q_{2,n}$ appears in \cite[Definition 4.1]{abinfty} and in \cite[Definition 5.1]{abdon-hurwitz}.
      Note that in theses papers is was called $Q_{2,n}$.

      \item
      If $\widetilde H_{1,n}$ is invertible so is the Schur complement ${\widehat {\widetilde H}}_{1,n}$. Consequently, the polynomial $P_{1,n+1}(0)$ is invertible if $\widetilde H_{1,n}$ is invertible.
   \end{enumerate}
\end{remark}

\begin{lem}\label{lem001}
Let $P_{1,n}$ be as in \eqref{Ppr} and set
\begin{equation}\label{d2m1}
  d^{(2n+1)}:=P_{1,n+1}^*(0)P_{1,n+1}^{*-1}(a).
\end{equation}
Furthermore, let
$(s_j)_{j=0}^{2n+1}$ be a Hausdorff positive sequence on $[a,b]$
and let $H_{4,n}$, $R_n$, $v_n$, $u_n$, $u_{n,4}$,
$P_{2,n}$, $Q_{1,n}$,  $Q_{2,n}$
$M^{(2n+1)}$ and $N^{(2n+1)}$ be as in \eqref{201A}, \eqref{52},
\eqref{59},
\eqref{69yy},
\eqref{22}, \eqref{Ppr}, \eqref{Qq1}, \eqref{Qq2} and \eqref{MN4n}, respectively. 
Recall that by Assumption~\ref{ass:H1tilde} the matrix
$\widetilde H_{1,n}$ is invertible.
Then the following equalities hold.
\begin{align}
   &d^{(2n+1)}=I-av_n^* H_{4,n}^{-1}R_n(a)u_{4,n},\label{eqW0}\\
   &\dstarinvodd= I+a u_n^*\widetilde H_{1,n}^{-1}R_n(a)v_n,\label{eqW01}\\
   &M^{(2n+1)}= aQ_{1,n+1}^*(0)P_{1,n+1}^{*-1}(0)\label{eqW1}\\
   &M^{(2n+1)}d^{(2n+1)}=a Q_{1,n+1}^*(0)P_{1,n+1}^{*-1}(a),\label{eqW3}
   \\
   &N^{(2n+1)}\dstarinvodd = -b v_n^* H_{3,n}^{-1}R_n(a) v_n,\label{eqW6}
   \\
   &N^{(2n+1)}\dstarinvodd =-b P_{2,n}^*(0)Q_{2,n}^{*-1}(a),\label{eqW5}
   \\
   &N^{(2n+1)}=-bP_{2,n}^*(0)Q_{2,n}^{*-1}(a)P_{1,n+1}^{-1}(a)P_{1,n+1}(0),
   \label{eqW2}
   \\
   &(I+M^{(2n+1)}N^{(2n+1)}) \dstarinvodd =Q_{2,n}^*(0)Q_{2,n}^{*-1}(a).\label{eqW4}
\end{align}
\end{lem}
\begin{proof} 
  First, we observe that
  \begin{align}
     \nonumber
     P_{1,n+1}(0)^* - P_{1,n+1}^*(a)
     & = v_{n+1}^* \big[ I - R_{n+1}^*(a) \big] \Sigma_{1, n+1}
     = v_{n+1}^* a T_{n+1}^* R_{n+1}^*(a) \Sigma_{1, n+1}
     \\
     \nonumber
     & = -a v_{n}^* L_{2, n+1}^* T_{n+1}^* R_{n+1}^*(a) \Sigma_{1, n+1}
     \\
     \label{eq:Pa-P0}
     & = -a v_{n}^* R_{n}^*(a) L_{1, n+1}^* \Sigma_{1, n+1}
  \end{align}
  where we used that
  $L_{2, n+1}^* T_{n+1}^* R_{n+1}^*(a) = L_{1, n+1}^* R_{n+1}^*(a) = R_{n}^*(a) L_{1, n+1}^*$ 
  by \eqref{eq:LTL1} and \eqref{eqV91}.
  \smallskip

  \noindent
  To show \eqref{eqW0} we use \eqref{eq:Pa-P0} and \eqref{eq502} to obtain
  \begin{align*}
     P_{1,n+1}^*(0) - P_{1,n+1}^*(a)
     & = -a v_{n}^* H_{4,n}^{-1} H_{4,n} R_{n}^*(a) L_{1, n+1}^* \Sigma_{1, n+1}
     \\
     & = -a v_{n}^* H_{4,n}^{-1} R_n(a)u_{4,n} v_{n+1}^*  R_{n+1}^*(a) \Sigma_{1, n+1}
     \\
     & = -a v_{n}^* H_{4,n}^{-1} u_{n} P_{1,n+1}^*(a),
  \end{align*}
  hence 
  $d^{(2n+1)} = P_{1, n+1}^*(0) P_{1, n+1}^{*-1}(a)
     = I - a v_{n}^* H_{4,n}^{-1} u_{n} P_{1,n+1}^*(a).  $
  To prove \eqref{eqW01}, we note that 
  by \eqref{eqV4} and \eqref{eq:lastcolumn}
  \begin{align}
     \label{eq:L1tildeH}
     \Sigma_{1, n+1}^* L_{1, n+1} 
     = \Sigma_{1, n+1}^* \big[ v_{n+1} u_n^* + H_{1, n+1}L_{2, n+1} \big] \widetilde H_{1,n}^{-1}
     = \Sigma_{1, n+1}^* v_{n+1} u_n^* \widetilde H_{1,n}^{-1}.
  \end{align}
  Hence \eqref{eqW01} follows from \eqref{eq:Pa-P0} because
  \begin{align*}
     P_{1,n+1}(0) - P_{1,n+1}(a)
     & = -a \Sigma_{1, n+1}^* L_{1, n+1} R_{n}(a) v_{n}
     \\
     & = -a \Sigma_{1, n+1}^* v_{n+1} u_n^* \widetilde H_{1,n}^{-1} R_{n}(a) v_{n}
     \\
     & = -a P_{1,n+1}(0) u_n^* \widetilde H_{1,n}^{-1} R_{n}(a) v_{n}.
  \end{align*}
  Formula \eqref{eqW1} follows from \eqref{eq:L1tildeH} and
  \begin{align*}
     aQ_{1,n+1}^*(0)P_{1,n+1}^{*-1}(0)
     &=-au_{1,n+1}^* \Sigma_{1, n+1} \big[ v_{n+1}^* \Sigma_{1, n+1} \big]^{-1}
     \\
     &=-au_{n}^* L_{1,n+1}^* \Sigma_{1, n+1} \big[ v_{n+1}^* \Sigma_{1, n+1} \big]^{-1}
     \\
     &=-au_{n}^* \widetilde H_{1,n}^{-1} u_n v_{n+1}^* \Sigma_{1, n+1}
     \big[ v_{n+1}^* \Sigma_{1, n+1} \big]^{-1}
     \\
     &=-au_{n}^* 
     \widetilde H_{1,n}^{-1} u_n 
     = M^{(2n+1)}.
  \end{align*}
  Equality \eqref{eqW3} readily follows from \eqref{eqW1} and \eqref{d2m1}.
  Next we prove \eqref{eqW6} using \eqref{eq:H4H1H3} and \eqref{eqW01}:
  \begin{align*}
     N^{(2n+1)} \dstarinvodd
     & = 
     -b v_n^* H_{3,n}^{-1} \widetilde H_{1,n} H_{4,n}^{-1} v_n
     \big[ I + a u_n^* \widetilde H_{1,n}^{-1} R_n(a) v_n \big]
     \\
     & = 
     -b v_n^* H_{3,n}^{-1} \widetilde H_{1,n} H_{4,n}^{-1} 
     \big[ I + a v_n u_n^* \widetilde H_{1,n}^{-1} R_n(a) \big] v_n
     \\
     & = 
     -b v_n^* H_{3,n}^{-1} \widetilde H_{1,n} H_{4,n}^{-1} 
     \big[ (I - T_n(a)) \widetilde H_{1,n} + a v_n u_n^* \big] \widetilde H_{1,n}^{-1} R_n(a) v_n
     \\
     & = 
     -b v_n^* H_{3,n}^{-1} \widetilde H_{1,n} H_{4,n}^{-1} H_{4,n} \widetilde H_{1,n}^{-1} R_n(a) v_n
     \\
     & = 
     -b v_n^* H_{3,n}^{-1} R_n(a) v_n
  \end{align*}
  where we used \eqref{eqV670} in the second to last step.

  Equality \eqref{eqW5} follows from \eqref{eqW6} and \cite[Equality (3.28)]{abKN}
  and Equality \eqref{eqW2} is an immediate consequence of \eqref{eqW5} and the definition of $d^{(2n+2)}$.

  Finally, we prove equality \eqref{eqW4} using \eqref{eqW6} and \eqref{eqW01}:
  \begin{align*}
     (I + M^{(2n+1)} N^{(2n+1)}) \dstarinvodd
     & = I + au_n^* \widetilde H_{1,n}^{-1} R_n(a) v_n
     + ab u_n^* \widetilde H_{1,n}^{-1} u_n
     v_n^* H_{3,n}^{-1} R_n(a) v_n
     \\
     & = I + a u_n^* \widetilde H_{1,n}^{-1} 
     \big[ H_{3,n} + b u_n v_n^* \big]
     H_{3,n}^{-1} R_n(a) v_n 
     \\
     & = I + a u_n^* (1-bT^*) H_{3,n}^{-1} R_n(a) v_n 
     \\
     & = I + a u_{3,n}^* H_{3,n}^{-1} R_n(a) v_n 
     \\
     & = Q_{2,n}^*(0)Q_{2,n}^{*-1}(a)
  \end{align*}
  where we used that $H_{3,n} + b u_n v_n^* = \widetilde H_{1,n}(I-bT_{n}^*)$ by \eqref{eqV672}
  and the last equality is a consequence of \eqref{eq11m1} and \cite[Equation (3.26)]{abKN}.
\end{proof}

Expressing the above equalities as orthogonal polynomials, we obtain the next identities.
\begin{corollary}\label{rem00BB}
 Let $P_{r,n}$ and $Q_{r,n}$ be as in Definition \ref{de002AAA}. 
 Then the following equality holds
 \begin{align}
    \label{eq:rem00B:1}
    P_{1,n+1}(a)Q_{1,n+1}^*(a)-a b Q_{1,n+1}(0)P_{2,n}^*(0)-P_{1,n+1}(0)Q_{2,n}^*(0)=0, \\
    \label{eq:rem00B:2}
    (b-a)Q_{4,n}(a)P_{3,n}^*(a) - b Q_{4,n}(0)P_{3,n}^*(0) - a P_{4,n}(0)Q_{3,n}^*(0) = 0.
 \end{align}
\end{corollary}
\begin{proof}
   Equation \eqref{eq:rem00B:1} follows from \eqref{eqW4}, \eqref{d2m1}, \eqref{eqW1} and \eqref{eqW5} and Remark~\ref{rem:Q0P0}, while \eqref{eq:rem00B:2} is a direct consequence of \eqref{eqn363}.
\end{proof}

Let
\begin{align}
{\mathfrak D^{(2n+1)}}:=
\begin{pmatrix}
   \dstarinvodd & 0 \\[1ex]
   0 & d^{(2n+1)}
\end{pmatrix}\label{D2nm1}
\end{align}
where $d^{(2n+1)}$ is defined in \eqref{d2m1}


\begin{theorem}\label{thmain}
Let  $U^{(2n+1)}(z)$, $V^{(2n+1)}$, ${\mathfrak J}_q$,  and ${\mathfrak D^{(2n+1)}}$
be matrices as in Definition \ref{rmodd2n1}, Definition \ref{def2001BB},
  \eqref{eqn432}
  and \eqref{D2nm1}, respectively.
Then the following equality holds:
             \begin{align}\label{eqn434m1}
U^{(2n+1)}(z)={\mathfrak J}_qV^{(2n+1)}(z){\mathfrak J}_q {\mathfrak D^{(2n+1)}}.
\end{align}
\end{theorem}

\begin{proof}
   Using the 
   block entries \eqref{eq11m1}--\eqref{eq22m1}, \eqref{eq11ar}--\eqref{eq22ar},
   \eqref{eqn:VV14m1} and
   \eqref{eq17},
   we see that Equality \eqref{eqn434m1} is equivalent to the following four equalities:
   \begin{align}
      \alpha^{(2n+1)}(z) &= \big[ \widetilde \gamma_4^{(n)}(z)N^{2n+1}+\widetilde \delta_4^{(n)}(z)(I+M^{(2n+1)}N^{(2n+1)}) \big] \dstarinvodd, \label{dem1m}
      \\
      \beta^{(2n+1)}(z) &= (z-a)^{-1} \big[ \widetilde \gamma_4^{(n)}(z)+
      \widetilde \delta_4^{(n)}(z)M^{(2n+1)} \big] d^{(2n+1)}, \label{dem2m}
      \\
      \gamma^{(2n+1)}(z) &= (z-a)\big[ \widetilde \alpha_4^{(n)}(z)N^{(2n+1)}
      +\widetilde \beta_4^{(n)}(I+M^{(2n+1)}N^{(2n+1)})\big] \dstarinvodd ,
      \label{dem3m}\\
      \delta^{(2n+1)}(z) &= \big[ \widetilde \alpha_4^{(n)}(z)
      +\widetilde \beta^{(2n+1)}(z)M^{(2n+1)}\big] d^{(2n+1)}. \label{dem4m}
   \end{align}

   For the proof of \eqref{dem1m} and \eqref{dem3m} we observe that, by \eqref{eq362a},
   \begin{align*}
      R_n^*({\bar z}) H_{4,n}^{-1} \left\{ -b u_{4,n} v_n^* + v_n u_{2,n}^* \right\} \Sigma_{2,n}
      = R_n^*({\bar z}) (I-bT_n^*) \Sigma_{2,n}
      = (I-bT_n^*)R_n^*({\bar z}) \Sigma_{2,n}.
   \end{align*}
   Hence it follows from \eqref{eqW5} and \eqref{eqW4} that
   \begin{multline*}
      \qquad
      \widetilde \gamma_4^{(n)}(z)N^{(2n+1)}\dstarinvodd 
      + \widetilde \delta_4^{(n)}(z) \big[ I+M^{(2n+1)}N^{(2n+1)} \big] \dstarinvodd 
      \\
      \begin{aligned}
	 &= \big[ -b\widetilde \gamma_4^{(n)}(z)P_{2,n}^*(0)
	 +\widetilde \delta_4^{(n)}(z)Q_{2,n}^*(0) \big] 
	 Q_{2,n}^{*-1}(a)
	 \\
	 &= \big[ -b\widetilde \gamma_4^{(n)}(z)v_n^*
	 -\widetilde \delta_4^{(n)}(z) \widehat u_{2,n}^* \big] 
	 \Sigma_{2,n} Q_{2,n}^{*-1}(a)
	 \\
	 &= \big[ 
	 z u_{4,n}^* R_n^*(\bar z) H_{4,n}^{-1} \left\{ -bu_{4,n}v_n^* + v_n \widehat u_{2,n}^* \right\}
	 - \widehat u_{2,n}^*
	 \big] 
	 \Sigma_{2,n} Q_{2,n}^{*-1}(a)
	 \\
	 &= \big[ 
	 z u_{4,n}^* (I - bT_n^*) R_n^*(\bar z) - \widehat u_{2,n}^*
	 \big] 
	 \Sigma_{2,n} Q_{2,n}^{*-1}(a)
	 \\
	 &= \big[ 
	 z u_{4,n}^* (I - bT_n^*) - \widehat u_{2,n}^* (I-zT_n^*)
	 \big] 
	 R_n(\bar z)^* \Sigma_{2,n} Q_{2,n}^{*-1}(a)
	 \\
	 &= -u_{2,n}^* R_n^*(\bar z) \Sigma_{2,n} Q_{2,n}^{*-1}(a)
	 = Q^*_{2,n}(\bar z) Q_{2,n}^{*-1}(a)
	 = \alpha^{(2n+1)}(z)
      \end{aligned}
   \end{multline*}
   which proves \eqref{dem1m}.
   The proof of \eqref{dem3m} is similar:
   \begin{multline*}
      (z-a) \big[ \widetilde \alpha_4^{(n)}(z)N^{(2n+1)}
      +\widetilde \beta_4^{(n)}(I+M^{(2n+1)}N^{(2n+1)}) \big]\dstarinvodd 
      \\
      \begin{aligned}
	 &=(z-a) \big[ -b \widetilde \alpha_4^{(n)}(z)P_{2,n}^*(0) + \widetilde \beta_4^{(n)}(z)Q^*_{2,n}(0) \big] 
	 Q_{2,n}^{*-1}(a)
	 \\
	 &=(z-a) \big[ -b \widetilde \alpha_4^{(n)}(z)v_n^* - \widetilde \beta_4^{(n)}(z) \widehat u_{2,n}^* \big] 
	 \Sigma_{2,n} Q_{2,n}^{*-1}(a)
	 \\
	 &=(z-a) \big[ 
	 -b v_n^*
	 + z v_n^* R_n^*(\bar z) H_{4,n}^{-1} \left\{ -b u_{4,n} v_n^* + v_n \widehat u_{2,n}^* \right\}
	 \big] 
	 \Sigma_{2,n} Q_{2,n}^{*-1}(a)
	 \\
	 &= (z-a) \big[ 
	 -b v_n^* + z v_n^* (I-bT_n^*)R_n^*(\bar z) 
	 \big] 
	 \Sigma_{2,n}Q_{2,n}^{*-1}(a)
	 \\
	 &= (z-a) \big[ 
	 -b v_n^*(I-zT_n^*) + z v_n^* (I-bT_n^*)
	 \big] 
	 R_n^*(\bar z) \Sigma_{2,n}Q_{2,n}^{*-1}(a)
	 \\
	 &= (z-a)(-b+z) v_n^* R_n^*(\bar z) \Sigma_{2,n}Q_{2,n}^{*-1}(a)
	 \\
	 &= -(z-a)(-b+z) Q_{2,n}^*(\bar z)Q_{2,n}^{*-1}(a)
	 = \gamma^{(2n+1)}(z).
      \end{aligned}
   \end{multline*}

   To prove \eqref{dem2m} and \eqref{dem4m} we first note that, by \eqref{eqV10},
   \begin{equation*}
      R_n^*(\bar z) H_{4, n}^{-1} \left\{ u_{4,n} v_{n+1}^* + a v_n u_{1, n+1}^* \right\} \Sigma_{1,n+1}
      = R_n^*(\bar z) L_{1,n+1}^* \Sigma_{1,n+1}
      = L_{1,n+1}^* R_{n+1}^*(\bar z) \Sigma_{1,n+1}.
   \end{equation*}
   Moreover,
   $u_{4,n}^* L_{1, n+1}^*
   = u_{n}^* (I-aT_n^*) L_{1,n+1}^*
   = u_{n}^* L_{1,n+1}^* (I-aT_{n+1}^*) 
   = u_{1, n+1}^* (I-aT_{n+1}^*)$
   by \eqref{eq:LTLT} and \eqref{eqV91}.
   Hence we obtain \eqref{dem2m}, using \eqref{d2m1} and \eqref{eqW3}, as follows:
   \begin{multline*}
      (z-a)^{-1}\big[
      \widetilde \gamma_4^{(n)}(z)d^{(2n+1)}+
      \widetilde \delta_4^{(n)}(z)  M^{(2n+1)}   d^{(2n+1)}
      \big]
      \\
      \begin{aligned}
	 \qquad
	 & =(z-a)^{-1}\big[
	 \widetilde \gamma_4^{(n)}(z) P_{1,n+1}^*(0) + a \widetilde \delta_4^{(n)}(z)Q_{1,n+1}^*(0) \big] P_{1,n+1}^{*-1}(a) 
	 \\
	 & =(z-a)^{-1}\big[
	 \widetilde \gamma_4^{(n)}(z) v_{n+1}^* - a \widetilde \delta_4^{(n)}(z) u_{1,n+1}^* \big] 
	 \Sigma_{1,n+1} P_{1,n+1}^{*-1}(a) 
	 \\
	 & =(z-a)^{-1}\big[
	 z u_{4,n}^* R_n^*(\bar z) H_{4,n}^{-1} \left\{ u_{4,n} v_{n+1}^* + a v_n u_{1, n+1}^* \right\}
	 - a u_{1, n+1}^* \big]
	 \Sigma_{1,n+1} P_{1,n+1}^{*-1}(a) 
	 \\
	 & =(z-a)^{-1}\big[
	 z u_{4,n}^* L_{1,n+1}^* R_{n+1}(\bar z)^* 
	 - a u_{1, n+1}^* \big]
	 \Sigma_{1,n+1}P_{1,n+1}^{*-1}(a) 
	 \\
	 & =(z-a)^{-1}\big[
	 z u_{1, n+1}^* (I - a T_{n+1}^*) 
	 - a u_{1, n+1}^*  (I - z T_{n+1}^*) \big]
	 R_{n+1}(\bar z)^* 
	 \Sigma_{1,n+1}P_{1,n+1}^{*-1}(a) 
	 \\
	 & = u_{1, n+1}^* R_{n+1}^*(\bar z) 
	 \Sigma_{1,n+1}P_{1,n+1}^{*-1}(a) 
	 = Q_{1, n+1}^*(\bar z) P_{1,n+1}^{*-1}(a) .
      \end{aligned}
   \end{multline*}
   In the third equality, we employed \eqref{eq21ar} and \eqref{eq22ar}.
   The proof of \eqref{dem4m} is similar:
   \begin{multline*}
      \qquad
      \widetilde \alpha_4^{(n)}(z)d^{(2n+1)}
      -\widetilde \beta_4^{(n)}(z)M^{(2n+1)}d^{(2n+1)}
      \\
      \begin{aligned}[b]
	 & = \big[ \widetilde \alpha_4^{(n)}(z) P_{1,n+1}^*(0)
	 -a \widetilde \beta_4^{(n)}(z) Q_{1,n+1}^*(0) \big] 
	 P_{1,n+1}^{*-1}(a)
	 \\
	 & = \big[ v_{n+1}^*
	 + z v_n^* R_n^*(\bar z) H_{4,n}^{-1}
	 \left\{ u_{4,n} v_{n+1}^* + a v_n u_{1, n+1}^* \right\}
	 \big]
	 \Sigma_{1,n+1} P_{1,n+1}^{*-1}(a)
	 \\
	 & = \big[ v_{n+1}^* + z v_n^* L_{1,n+1}^* R_{n+1}^*(\bar z) 
	 \big]
	 \Sigma_{1,n+1} P_{1,n+1}^{*-1}(a)
	 \\
	 & = \big[ v_{n+1}^* (I-zT_{n+1}^*) + z v_n^* L_{1,n+1}^* 
	 \big]
	 R_{n+1}(\bar z)^* 
	 \Sigma_{1,n+1} P_{1,n+1}^{*-1}(a)
	 \\
	 & = v_{n+1}^* R_{n+1}(\bar z)^* \Sigma_{1,n+1} P_{1,n+1}^{*-1}(a)
	 = P_{1,n+1}(\bar z)^*  P_{1,n+1}^{*-1}(a).
      \end{aligned}
      \qedhere
   \end{multline*}
\end{proof}

\subsection{Case of an odd number of moments} 
In this subsection, we prove the relation between of the resolvent matrix
proposed in \cite{D-C} and \cite{abdon2} in the case of an odd number of moments.
\smallskip

First, we reproduce the resolvent matrix from \cite{abdon2} for the case of an odd number of moments.

\begin{assumption}
   \label{ass:Gamma}
   In this section we will always assume that
   \begin{equation*}
      \tag{$\Gamma$}
      I+a v_n^* R_n^*(a)H_{1,n}^{-1}u_{1,n}\qquad \text{is invertible.}
   \end{equation*}
\end{assumption}

\begin{definition}\label{defA01}
Let us define the $q\times q$ matrices
\begin{align}
   \Gamma_a:=&
   \left(I+a v_n^* R_n^*(a)H_{1,n}^{-1}u_{1,n} \right)^{-1}
   v_n^* R_n^*(a)H_{1,n}^{-1} v_n,\label{gammaa}\\
   \Gamma_b:=&
   \left(I+b v_n^* R_n^*(b)H_{1,n}^{-1}u_{1,n} \right)^{-1}
   v_n^* R_n^*(b)H_{1,n}^{-1} v_n \label{gammab}
\end{align}
and
\begin{align}
   M^{(2n)}:=
   a \Gamma_a,\qquad
   N^{(2n)} :=\left(b\Gamma_b -a \Gamma_a\right)^{-1}. \label{eqnNM2n}
\end{align}
Moreover, we define the $2q\times 2q$ matrices 
\begin{align}
   \label{eq:CD}
   C_{1}^{(2n)}&:=
   \begin{pmatrix}
      I       & M^{(2n)}                \\
      0   & I
   \end{pmatrix},
   \quad
   D_{1}^{(2n)}:=
   \begin{pmatrix}
      I       & 0                \\
      N^{(2n)}   & I
   \end{pmatrix},
\end{align}
and
\begin{align}\label{eqn:VV14}
   V^{(2n)}(z):=&\widetilde V_{1}^{(n)}(z)
   C_{1}^{(2n)}D_{1}^{(2n)}
\end{align}
The matrix \eqref{eqn:VV14} is called the 
\define{first auxiliary resolvent matrix in the point $0$ of the THMM problem in the case of an even number of moments}.
\end{definition}

\begin{remark}
   \begin{enumerate}[(i)]
      \item
      Note that $a\Gamma_a = \widetilde\alpha_1^{n}(a)^{-1}\widetilde\beta_1^n(a)$
      and $b\Gamma_b = \widetilde\alpha_1^{n}(b)^{-1}\widetilde\beta_1^n(b)$.

      \item 
      We will show in Lemma~\ref{lem4.2} that $b\Gamma_b  - a\Gamma_a$ is invertible, hence $N^{(2n)}$ is well-defined.

   \end{enumerate}
\end{remark}

Now we recall the resolvent matrix of the THMM problem for the
case of an odd number of moments given in terms of orthogonal matrix polynomials
\cite{abKN}.
\begin{definition}[\protect{\cite[Theorem 3.5]{abKN}}] 
   \label{rmodd2}
   Let $(s_j)_{j=0}^{2n}$
   be a Hausdorff positive definite sequence on $[a,b]$.
   Let the matrices $P_{3,n}$, $P_{4,n}$, $Q_{3,n}$ and $Q_{4,n}$ be as in Definition \ref{de002AAA}.
   The $2q\times 2q$ matrix polynomial
   \begin{equation}
      U^{(2n)}(z):=
      \begin{pmatrix}
	 \alpha^{(2n)}(z) & \beta^{(2n)}(z)\\
	 \gamma^{(2n)} (z)& \delta^{(2n)}(z)
      \end{pmatrix},\quad z\in {\mathbb C},
      \label{RM1}
   \end{equation}
   with
   \begin{align}
      \alpha^{(2n)}(z):=&Q_{4,n}^*(\bar z)Q_{4,n}^{*-1}(a),
      \label{eq11}\\
      \beta^{(2n)}(z):=&\frac{1}{b-a}Q_{3,n}^*(\bar z)P_{3,n}^{*-1}(a),
      \label{eq12}\\
      \gamma^{(2n)}(z):=&(z-a)P_{4,n}^*(\bar z)Q_{4,n}^{*-1}(a),
      \label{eq21}\\
      \delta^{(2n)}(z):=&\frac{b-z}{b-a}P_{3,n}^*(\bar z)P_{3,n}^{*-1}(a).
      \label{eq22}
\end{align}
 is called the \define{resolvent matrix of the THMM problem in the point $a$ in the case of an odd number
 of moments}.
\end{definition}
\begin{remark}\label{rem00C}
   Let $Q_{4,n}$ be defined as in \eqref{Qq4}.
   Then the matrix $Q_{4,n}(0)$ is invertible.
\end{remark}
\begin{proof}
   Recall that the matrix $Q_{4,n}(a)$ is invertible by Remark~\ref{remPaQa} and that
   equality \eqref{pqH2B} gives for $z=0$ 
   \begin{equation}\label{eqQaa}
      I+ a v_n^*R_n^*(a)H_{1,n}^{-1} u_{1,n}=Q_{4,n}^{-1}(a)Q_{4,n}(0).
   \end{equation}
   Since by Assumption~\ref{ass:Gamma} the left hand side of \eqref{eqQaa} is invertible,
   the matrix $Q_{4,n}(0)$ is invertible.
\end{proof}

\begin{definition}
   Let $Q_{4,n}$ be as in \eqref{Qq4}. 
   We set
   \begin{equation}
      d^{(2n)}:=Q_{4,n}^*(0)Q_{4,n}^{*-1}(a).\label{eqnd11}
   \end{equation}

\end{definition}

\begin{lem}\label{lem4.2} Let $P_{k,n}$, $Q_{k,n}$ for $k=3,4$, $\Gamma_a$, $\Gamma_b$,
  $M^{(2n)}$ and $N^{(2n)}$ be as in Definition~\ref{de002AAA}, \eqref{gammaa}, \eqref{gammab} and \eqref{eqnNM2n}, respectively.
  Then the following equalities are valid.
  \begin{align}
     \Gamma_a &= -P_{4,n}^*(0)Q_{4,n}^{*-1}(0), \label{eqn38A}\\
     \Gamma_b &= P_{3,n}^*(0)Q_{3,n}^{*-1}(0), \label{eqn39A}\\
     \Gamma_a &= \Gamma_a^{*}, \label{eqn43A}\\
     \Gamma_b &= \Gamma_b^{*}, \label{eqn44A}\\
     M^{(2n)} &= -a P_{4,n}^*(0)Q_{4,n}^{*-1}(0), \label{eqn45a}\\
     N^{(2n)} &= (b-a)^{-1}Q_{3,n}^*(0) P_{3,n}^{*-1}(a)Q_{4,n}^{-1}(a)Q_{4,n}(0),
     \label{eqn001}
     \\
     \label{eqn101}
     N^{(2n)} \dstarinveven &= (b-a)^{-1} Q_{3,n}(0)^* P_{3,n}^{*-1}(a),
     \\
     \label{eqn102}
     (I + M^{(2n)} N^{(2n)}) \dstarinveven 
     &= 
     (b-a)^{-1} b P_{3,n}^*(0)P_{3,n}^{*-1}(a).
  \end{align}
\end{lem}
\begin{proof} 
   The equality \eqref{eqn38A} was shown in the proof of Remark~\ref{rem00C}.
   To prove \eqref{eqn43A} and \eqref{eqn44A} it suffices to note that 
   $P_{r,n}(0)^* [Q_{r,n}(0)^*]^{-1} = [Q_{r,n}(0)]^{-1} P_{r,n}(0)$
   which is an immediate consequence of Remark~\ref{rem:Q0P0}.
   Equation \eqref{eqn45a} follows directly from the definition of $M^{(2n)}$ and \eqref{eqn38A}.
   For the proof of \eqref{eqn001} note that by \eqref{eqn38A}, \eqref{eqn39A} and the selfadjointness of $\Gamma_a$ we have that
   \begin{align*}
      a\Gamma_a - b\Gamma_b
      & = -  a Q_{4,n}(0)^{-1} P_{4,n}(0)  - b P_{3,n}(0)^* Q_{3,n}(0)^{*^-1} 
      \\
      & = -  Q_{4,n}(0)^{-1} \left[ a P_{4,n}(0) Q_{3,n}(0)^* - b Q_{4,n}(0) P_{3,n}(0)^* \right] Q_{3,n}(0)^{*-1} 
      \\
      & = -  Q_{4,n}(0)^{-1} \Sigma_{4,n}^* \left[ a v_{n}^* u_{3,n}^* - b u_{4,n} v_{n}^* \right] \Sigma_{3,n} Q_{3,n}(0)^{*-1} 
      \\
      & = -(a-b) (Q_{4,n}(0)^{-1} \Sigma_{4,n}^* R_n(a) u_{4,n} v_n^* R_n(a)^* \Sigma_{3,n} Q_{3,n}(0)^{*-1} 
      \\
      &= (b-a) (Q_{4,n}(0)^{-1} Q_{4,n}(a)^* P_{3,n}(a)^*  Q_{3,n}(0)^{*-1}.
   \end{align*}
   By Remark~\ref{remPaQa}, the matrices $Q_{4,n}(a)^*$ and $P_{3,n}(a)^*$ are invertible, hence $N^{(2n)}$ is well-defined and \eqref{eqn001} holds.
   Now also \eqref{eqn101} is clear.
   Finally, we show \eqref{eqn102}.
   \begin{multline*}
     (I + M^{(2n)} N^{(2n)}) \dstarinveven 
     = \left[ [N^{(2n)}]^{-1} + M^{(2n)} \right]  N^{(2n)}\dstarinveven 
     = b\Gamma_b  N^{(2n)}\dstarinveven 
     \\
     = (b-a)^{-1} b P_{3,n}^*(0)Q_{3,n}^{*-1}(0) Q_{3,n}^{*}(0)P_{3,n}^{*-1}(a)
      = (b-a)^{-1} b P_{3,n}^*(0)P_{3,n}^{*-1}(a).
   \end{multline*}
\end{proof}

With $d^{(2n)}$ as in \eqref{eqnd11}, we define
 \begin{align}
    \label{D2n}
    {\mathfrak D^{(2n)}}:=
    \begin{pmatrix}
       d^{(2n)}& 0 \\[1ex]
       0 & \dstarinveven
    \end{pmatrix}.
 \end{align}

Now we formulate and prove the main result of this subsection.
\begin{theorem}\label{thmainodd}
Let  $U^{(2n)}(z)$, ${\mathfrak J}_q$, $V_1^{(2n)}$ and ${\mathfrak D^{(2n)}}$
be matrices as in \eqref{RM1}, \eqref{eqn432}, \eqref{eqn:VV14} and \eqref{D2n}, respectively.
Then
the following equality holds:
\begin{align}\label{eqn434}
   U^{(2n)}(z)={\mathfrak J}_qV^{(2n)}(z){\mathfrak J}_q {\mathfrak D^{(2n)}}.
\end{align}
\end{theorem}

\begin{proof}
  Using the
   block entries \eqref{eq11}--\eqref{eq22}, \eqref{eq11ar}--\eqref{eq22ar}, \eqref{eqn:VV14} and
   \eqref{eqnNM2n},
   we see that \eqref{eqn434} is equivalent to the following four equalities:
   \begin{align}
      \alpha^{(2n)}(z) 
      &= \left[ \widetilde \gamma_1^{(n)}(z)M^{(2n)}
      + \widetilde \delta_1^{(n)}(z)\right]d^{(2n)} , \label{dem1}
      \\
      \beta^{(2n)}(z) 
      &= \left[ \widetilde \gamma_1^{(n)}(z)(I+M^{(2n)}N^{(2n)})
      + \widetilde \delta_1^{(n)}(z)N^{(2n)} \right]\dstarinveven, \label{dem2}
      \\
      \gamma^{(2n)}(z) 
      &= \left[ \widetilde \alpha_1^{(n)}(z)M^{(2n)}
      + \widetilde \beta_1^{(n)} \right] d^{(2n)}, \label{dem3}
      \\
      \delta^{(2n)}(z) 
      &= \left[ \widetilde \alpha_1^{(n)}( z)(I+M^{(2n)}N^{(2n)})
      + \widetilde \beta_1^{(n)}( z)N^{(2n)} \right] \dstarinveven. \label{dem4}
   \end{align}
   For the proof of 
   \eqref{dem1} and \eqref{dem3} we will use that, by \eqref{eqn361}, 
   \begin{align*}
      R_n(\bar z)^* H_{1,n}^{-1} 
      \left\{ -a u_{1,n} v_n^* - v_n u_{4,n} \right\}
      \Sigma_{4,n}
      &= R_n(\bar z)^* (1-aT_n^*) \Sigma_{4,n}
      = (I-aT_n^*) R_n(\bar z)^* \Sigma_{4,n}.
   \end{align*}
   Therefore, by \eqref{eq11}, \eqref{eqn45a} and \eqref{eqnd11}, we have:
   \begin{align*}
      \left[ \widetilde \gamma_1^{(n)}(z)M^{(2n)} + \widetilde \delta_1^{(n)}( z) \right] d^{(2n)}
      &=\left[ 
      -a \widetilde \gamma_1^{(n)}(z)P_{4,n}(0)^* Q_{4,n}(0)^{*-1} 
      + \widetilde \delta_1^{(n)} \right] Q_{4,n}(0)^* Q_{4,n}(0)^{*-1}
      \\
      &= \left[ z u_{1,n}^* R_n(\bar z)^* H_{1,n}^{-1} 
      \left\{ -a u_{1,n} v_n^* - v_n u_{4,n} \right\}
      + u_{4,n}^* \right] 
      \Sigma_{4,n} Q_{4,n}(0)^{*-1}
      \\
      &= \left[ z u_{1,n}^* (1-aT_n^*) + u_{4,n}^* (1-zT_n^*) \right] 
      R_n(\bar z)^* \Sigma_{4,n} Q_{4,n}(0)^{*-1}
      \\
      &= \left[ z u_{4,n}^* T_n^* + u_{4,n}^* (I-zT_n^*) \right] 
      R_n(\bar z)^* \Sigma_{4,n} Q_{4,n}(0)^{*-1}
      \\
      &= u_{4,n}^* R_n(\bar z)^* \Sigma_{4,n} Q_{4,n}(a)^{*-1}
      = Q_{4,n}(\bar z)^*Q_{4,n}(a)^{*-1}
      = \alpha^{(2n)}(z)
   \end{align*}
   where we used that 
   $u_{1,n}^* (I-aT_n^*) = u_{n}^* (I-aT_n^*) T_n^* = u_{4,n}^* T_n^*$.
   Hence \eqref{dem1} is proved.
   For the proof of \eqref{dem1}, we calculate
   \begin{align*}
      \left[ \widetilde \alpha_1^{(n)}(z) M^{(2n)} + \widetilde \beta_1^{(n)}( z) \right] d^{(2n)}
      &=\left[ 
      -a \widetilde \alpha_1^{(n)}(z)P_{4,n}(0)^* Q_{4,n}(0)^{*-1} 
      + \widetilde \beta_1^{(n)} \right] Q_{4,n}(0)^* Q_{4,n}(0)^{*-1}
      \\
      &= \left[ -av_n^* 
      + z v_n^* R_n(\bar z)^* H_{1,n}^{-1} 
      \left\{ -a u_{1,n} v_n^* - v_n u_{4,n} \right\}
      \right] 
      \Sigma_{4,n} Q_{4,n}(a)^{*-1}
      \\
      &= \left[ -av_n^* (1-zT_n^*)
      + z v_n^* (I-aT_n^*) \right] 
      R_n(\bar z)^* \Sigma_{4,n} Q_{4,n}(a)^{*-1}
      \\
      &= (z-a) v_n^* R_n(\bar z)^* \Sigma_{4,n} Q_{4,n}(a)^{*-1}
      = (z-a) P_{4,n}(z)^* Q_{4,n}(a)^{*-1}
      \\
      &= \gamma^{(2n)}(z).
   \end{align*}
   
   For the proof of \eqref{dem2} and \eqref{dem2}, we will use that, by \eqref{eqn362}, the following holds:
   \begin{align*}
      R_n(\bar z)^* H_{1,n}^{-1} 
      \left\{ b u_{1,n} v_n^* - v_n u_{3,n}^* \right\}
      \Sigma_{3,n}
      &= -R_n(\bar z)^* (I-bT_n^*) \Sigma_{3,n}
      = -(I-bT_n^*) R_n(\bar z)^* \Sigma_{3,n}.
   \end{align*}

   Therefore, by \eqref{eqn102} and \eqref{eqn101} we have that 
   \begin{multline*}
      \left[ \widetilde \gamma_1^{(n)}( z)(I+M^{(2n)}N^{(2n)}) 
      +\widetilde \delta_1^{(n)}(z)N^{(2n)} \right]\dstarinveven
      \\
      \begin{aligned}
	 &= (b-a)^{-1} \left[
	 b\widetilde \gamma_1^{(n)}(z) P_{3,n}(0)^* P_{3,n}(a)^{*-1}
	 +\widetilde \delta_1^{(n)}(z) Q_{3,n}(0)^* P_{3,n}(a)^{*-1}
	 \right]
	 \\
	 &= (b-a)^{-1} \left[
	 b\widetilde \gamma_1^{(n)}(z) v_n^*
	 +\widetilde \delta_1^{(n)}(z) u_{3,n}^*
	 \right] \Sigma_{3,n} P_{3,n}(a)^{*-1}
	 \\
	 &= (b-a)^{-1} \left[
	 z u_{1,n}^* R_n(\bar z)^* H_{1,n}^{-1} 
	 \left\{ b u_{1,n} v_n^* - v_n u_{3,n}^* \right\}
	 + u_{3,n}^*
	 \right] \Sigma_{3,n} P_{3,n}(a)^{*-1}
	 \\
	 &= (b-a)^{-1} \left[
	 - z u_{1,n}^* (1-bT_n^*)
	 + u_{3,n}^* (I-zT_n^*)
	 \right] R_n(\bar z)^* \Sigma_{3,n} P_{3,n}(a)^{*-1}
	 \\
	 &= (b-a)^{-1} u_{3,n}^* R_n(\bar z)^* \Sigma_{3,n} P_{3,n}(a)^{*-1}
	 = (b-a)^{-1} Q_{3,n}(\bar z)^* P_{3,n}(a)^{*-1}
	 = \beta^{(2n)}(z)
      \end{aligned}
   \end{multline*}
   where we used that
   $u_{1,n}^* (I-bT_n^*) =u_{n}^* T_n^*(I-bT_n^*) = -u_{3,n}^* T_n^*$.
   Finally, we prove \eqref{dem4}. 
   Again we use \eqref{eqn102} and \eqref{eqn101} to obtain
   \begin{multline*}
      \left[ \widetilde \alpha_1^{(n)}(z)(I+M^{(2n)}N^{(2n)})
      + \widetilde \beta_1^{(n)}(z)N^{(2n)} \right] \dstarinveven
      \\
      \begin{aligned}[b]
	 &= (b-a)^{-1} \left[ 
	 b \widetilde \alpha_1^{(n)}(z) P_{3,n}^*(0) P_{3,n}(a)^{*-1}
	 + \widetilde \beta_1^{(n)}(z) Q_{3,n}(0)^* P_{3,n}(a)^{*-1}
	 \right]
	 \\
	 &= (b-a)^{-1} \left[ 
	 b  \widetilde \alpha_1^{(n)}(z) v_n^*
	 + \widetilde \beta_1^{(n)}(z) u_{3,n}^* 
	 \right]
	 \Sigma_{3,n} P_{3,n}(a)^{*-1}
	 \\
	 &= (b-a)^{-1} \left[ 
	 b v_n^*
	 + z v_n^*R_n(\bar z)^* H_{1,n}^{-1}
	 \left\{ b u_{1,n} v_n^*
	 - v_n u_{3,n}^* \right\}
	 \right]
	 \Sigma_{3,n} P_{3,n}(a)^{*-1}
	 \\
	 &= (b-a)^{-1} \left[ 
	 b v_n^*
	 - z v_n^* (1-bT_n^*) R_n(\bar z)^*
	 \right]
	 \Sigma_{3,n} P_{3,n}(a)^{*-1}
	 \\
	 &= (b-z)(b-a)^{-1} 
	 v_n^* R_n(\bar z)^* \Sigma_{3,n} P_{3,n}(a)^{*-1}
	 \\
	 &= (b-z)(b-a)^{-1} P_{3,n}(\bar z)^* P_{3,n}(a)^{*-1}
	 = \delta^{2n}(z).
   \end{aligned}
   \qedhere
   \end{multline*}
\end{proof}

\appendix

\section{Power series expansion of the resolvent matrices}

\subsection{Expansion in $z=0$} 

Let $(s_k)_{k=0}^{2n+1}$ be a Hausdorff positive definite sequence on $[a,b]$.
 Let $H_{r,n}$, $u_{r,n}$ for
$r=3, 4$, $R_n$ and $v_n$ be defined as in \eqref{201A},
\eqref{70uu}, \eqref{hatu2}, \eqref{21A}, \eqref{22}, \eqref{52} and \eqref{59}, respectively.
Furthermore, assume that $\widetilde H_{1,n}$ from \eqref{eqHt1n} is an invertible matrix.

According to Theorem~4 of \cite{D-C} the entries of the matrix 
\begin{align*}
   V^{(2n+1)} = 
   \begin{pmatrix}
      \widehat{\alpha}^{(2n+1)}(z) & \widehat{\beta}^{(2n+1)}(z) \\
      \widehat{\gamma}^{(2n+1)}(z) & \widehat{\delta}^{(2n+1)}(z)
   \end{pmatrix}
\end{align*}
from \eqref{eq17} are
\begin{align}
   \widehat{\alpha}^{(2n+1)}(z)&:=I+zv_n^{*}R_n^*(\bar z) \left( \frac {
   bH_{4,n}+aH_{3,n}}{b-a} \right)^{-1}  u_n,
   \label{eqn:20}\\
   \widehat{\gamma}^{(2n+1)}(z)&:=u_n^{*}R_n^*(\bar z) \left( \frac {b
   H_{4,n}+aH_{3,n}}{b-a} \right)^{-1}  u_n,
   \label{eqn:21}\\
   \widehat{\beta}^{(2n+1)}(z)&:=(z-b)(z-a)v_n^{*}R_n^*(\bar z)
   \frac {a H_{4,n}^{-1}+bH_{3,n}^{-1}}{b-a}v_n ,
   \label{eqn:22}\\
   \widehat{\delta}^{(2n+1)}(z)&:= I+u_n^{*}R_n^*(\bar z)
   \frac{a(z-b)R_n^{*{-1}}(a)H_{4,n}^{-1}+
   b(z-a)R_n^{*{-1}}(b)H_{3,n}^{-1}}{b-a}v_n.
   \label{eqn:23}
\end{align}

\begin{remark} \label{remsep1}
Using the Equalities \eqref{eqHtH34} and \eqref{eq:H3H4}
we find that
\begin{equation}\label{eqZZ1}
   aH_{4,n}^{-1}+bH_{3,n}^{-1}=(b-a)H_{4,n}^{-1}\widetilde H_{1,n}H_{3,n}^{-1}
\end{equation}
and
\begin{multline*}
   a(z-b)R_n^{*{-1}}(a)H_{4,n}^{-1}+ b(z-a)R_n^{* {-1}}(b)H_{3,n}^{-1}
   \\
   \begin{aligned}
      &=
      (b-a)ab\big[ T_n^*H_{4,n}^{-1}\widetilde H_{1,n}-H_{4,n}^{-1}H_{1,n} \big]
      H_{3,n}^{-1}
      \\
      &\phantom{=\ \ }+ z (b-a)\big[ H_{4,n}^{-1}\widetilde H_{1,n}+ T_n^* H_{4,n}^{-1}(abH_{1,n}-(a+b)\widetilde H_{1,n}) \big] H_{3,n}^{-1}.
   \end{aligned}
\end{multline*}

Therefore the entries in the matrix $V^{(2n+1)}$ can be expanded in powers of $z$ as follows:
\begin{align*}
   \widehat \alpha^{2n+1}(z)
   &= I + v_n^* \sum_{j=1}^{n-1} z^{j} T_n^{*(j-1)} \widetilde H_{1,n}^{-1} u_n,
   \\
   \widehat \gamma^{2n+1}(z) 
   & = u_n^* \sum_{j=0}^n z^{j} T_n^{*j} \widetilde H_{1,n}^{-1} u_n,
   \\
   \widehat \beta^{2n+1}(z) 
   &= 
   v_n^* \bigg\{ ab 
   + z \big[ - (a+b) + ab T_n^{*} \big]
   + \sum_{j=2}^{n+2} z^j T_n^{*(j-2)} \big[ I - (a+b)T_n^* + ab T_n^{*2} \big]  
   \bigg\}
   \\
   &\qquad\qquad \times H_{4,n}^{-1} \widetilde H_{1,n} H_{3,n}^{-1} v_n,
   \\
   \widehat \delta^{2n+1}(z) 
   &=
   u_n^* \bigg\{
   ab
   + z \big[ (a+b) + ab T_n^{*} \big]
   + \sum_{j=1}^{n+1}
   z^j T_n^{*(j-1)} \big[ I - (a+b) T_n^{*} + ab T_n^{*2} 
   \big]
   \bigg\}
   \\
   \widehat \delta^{2n+1}(z) 
   &= I + ab u_n^* 
   \big[ T_n^{*}H_{4,n}^{-1}\widetilde H_{1,n} - H_{4,n}^{-1} H_{1,n} \big] H_{3,n}^{-1}v_n
   \\ & \phantom{=\ }
   + u_n^* \bigg\{
   \sum_{j=1}^{n+1}
   z^j T_n^{*(j-1)} \big[ I - (a+b) T_n^{*} + ab T_n^{*2} 
   \big]
   \bigg\}
   H_{4,n}^{-1} \widetilde H_{1,n} H_{3,n}^{-1} v_n
\end{align*}
and the coefficients of the expansion in \eqref{v2nm1A} are for $2\le j \le n-1$

\begin{align*}
   \widetilde A_0 &= \begin{pmatrix}
      I&ab v_n^* H_{4,n}^{-1}\widetilde H_{1,n}H_{3,n}^{-1}v_n\\
      u_n^*\widetilde H_{1,n}^{-1}u_n&I
      +abu_n^*(T_n^*H_{4,n}^{-1}\widetilde H_{1,n}-H_{4,n}^{-1}H_{1,n})
      H_{3,n}^{-1}v_n
   \end{pmatrix},
   \\[2ex]
   \widetilde A_1 &=
   \begin{pmatrix}
      v_n^*\widetilde H_{1,n}^{-1}u_n &
      v_n^* \big[ abT_n^*-(a+b)I \big] H_{4,n}^{-1}\widetilde H_{1,n}H_{3,n}^{-1}v_n
      \\
      u_n^*T_n^*\widetilde H_{1,n}^{-1}u_n&
      u_n^* \big[ I - (a+b) T_n^* + ab (T_n^*)^2
      \big] H_{4,n}^{-1} \widetilde H_{1,n} H_{3,n}^{-1} v_n
   \end{pmatrix},
   \\[2ex]
   \widetilde A_j &=
   \begin{pmatrix}
      v_n^*T_n^{*(j-1)} \widetilde H_{1,n}^{-1} u_n
      &
      v_n^* T_n^{*(j-2)} \big[ I - (a+b)T_n^* + ab T_n^{*2} \big]  
      H_{4,n}^{-1} \widetilde H_{1,n} H_{3,n}^{-1} v_n
      \\
      u_n^* T_n^{*j} \widetilde H_{1,n}^{-1} u_n
      &
      u_n^* 
      T_n^{*(j-1)} \big[ I - (a+b) T_n^{*} + ab T_n^{*2} \big]
      H_{4,n}^{-1} \widetilde H_{1,n} H_{3,n}^{-1} v_n
   \end{pmatrix},
   \\
   &=
   \begin{pmatrix}
      v_n^* T_n^{*(j-2)} & 0 \\ 0 & u_n^* T_n^{*(j-1)}
   \end{pmatrix}
   \begin{pmatrix}
      I & I - (a+b)T_n^* + ab T_n^{*2}
      \\
      I & I - (a+b) T_n^{*} + ab T_n^{*2}
   \end{pmatrix}
   \begin{pmatrix}
      T_n^* \widetilde H_{1,n}^{-1} u_n & 0 \\
      0 & H_{4,n}^{-1} \widetilde H_{1,n} H_{3,n}^{-1} v_n
   \end{pmatrix}
   \\[2ex]
   \widetilde A_{n+1}
   &= \begin{pmatrix}
      v_n^*T_n^{*n}\widetilde H_{1,n}^{-1}u_n &  v_n^*T_n^{*(n-1)} \big[ I-(a+b)T_n^* \big] H_{4,n}^{-1}\widetilde H_{1,n}H_{3,n}^{-1}v_n\\
      0&
      u_n^* T_n^{*n} H_{4,n}^{-1}\widetilde H_{1,n}H_{3,n}^{-1}v_n
   \end{pmatrix},
   \\[2ex]
   \widetilde A_{n+2}
   &= \begin{pmatrix}
      0&  v_n^*T_n^{*n}H_{4,n}^{-1}\widetilde H_{1,n}H_{3,n}^{-1}v_n\\
      0&0
   \end{pmatrix}.
\end{align*}

\end{remark}

\begin{remark} \label{remsep2}
Let the matrices $C$, $D$ be as in \eqref{eq:CD}
and let $\widetilde V_r^{(n)}$ be as in \eqref{RM2nm1A} for $r=1$.
Then the matrices $\widetilde B_0$, $\widetilde B_1$, $\widetilde B_{n}$ and $\widetilde B_{n+1}$
in the representation \eqref{v2nAA} are
\begin{align*}
   \widetilde B_0 &= CD,
   \\
   \widetilde B_j &=
   \begin{pmatrix}
      v_n^* T_n^{*(j-1)} & 0 \\
      0 & u_{1,n}^* T_n^{*(j-1)}
   \end{pmatrix}
   \begin{pmatrix}
      I & -I \\
      I & -I
   \end{pmatrix}
   \begin{pmatrix}
      H_{1,n}^{-1} u_{1,n} & 0 \\
      0 & H_{1,n}^{-1} v_{n}
   \end{pmatrix}
   CD,
   \qquad 1\le j \le n+1
\end{align*}
because
\begin{align*}
   \widetilde \alpha_r^{(n)}(z) 
   &= I+z v_n^*R_n^*(\bar z)H_{r,n}^{-1}u_{r,n}
   = I + v_n^* \sum_{j=1}^{n+1} z^j T_n^{*(j-1)} H_{r,n}^{-1}u_{r,n},
   \\
   \widetilde\beta_r^{(n)}(z)
   &= -z v_n^*R_n^*(\bar z)H_{r,n}^{-1}v_n
   = - v_n^* \sum_{j=1}^{n+1} z^j T_n^{*(j-1)} H_{r,n}^{-1}v_{n},
   \\
   \widetilde\gamma_r^{(n)}(z)
   &= z u_{r,n}^*R_n^*(\bar z)H_{r,n}^{-1}u_{r,n}
   = u_{r,n}^* \sum_{j=1}^{n+1} z^j T_n^{*(j-1)} H_{r,n}^{-1}u_{r,n},
   \\
   \widetilde\delta_r^{(n)}(z)
   &= I-z u_{r,n}^*R_n^*(\bar z)H_{r,n}^{-1}v_n
   = I - u_{r,n}^* \sum_{j=1}^{n+1} z^j T_n^{*(j-1)} H_{r,n}^{-1}v_{n}.
\end{align*}
\end{remark}

\subsection{Expansion in $z=a$} 

\begin{remark}\label{remsep3}
   The resolvent matrix $U^{(2n+1)}$ from \eqref{RM2nm1} can be written
   as in \eqref{v2nm1C} where the coefficient matrices are for $2\le j\le n$ are
   \begin{align*}
      \widetilde C_0
      &= \begin{pmatrix}
	 I&
	 -Q_{1,n+1}^*(a)P_{1,n+1}^{*-1}(a)\\
	 0&I
      \end{pmatrix},
      \\[2ex]
      \widetilde C_1
      &= \begin{pmatrix}
	 Q_{2,n}^{*\prime}(a)Q_{2,n}^{*-1}(a) &
	 -Q_{1,n+1}^{*\prime}(a)P_{1,n+1}^{*-1}(a)
	 \\
	 -(b-a)P_{2,n}^*(a)Q_{2,n}^{*-1}(a)&
	 P_{1,n+1}^{*\prime}(a)P_{1,n+1}^{*-1}(a)
      \end{pmatrix},
      \\[2ex]
      \widetilde C_j
      &= \frac{1}{j!} 
      \begin{pmatrix}
	 Q_{2,n}^{*[j]}(a)Q_{2,n}^{*-1}(a) 
	 & -Q_{1,n+1}^{*[j]}(a)P_{1,n+1}^{*-1}(a)
	 \\
	 j\big[ (j-1)P_{2,n}^{*[j-2]}(a) -  (b-a) P_{2,n}^{*[j-1]}(a) \big] Q_{2,n}^{*-1}(a)
	 & P_{1,n+1}^{*[j]}(a)P_{1,n+1}^{*-1}(a)
      \end{pmatrix}
      \\
      &= \frac{1}{j!} 
      \begin{pmatrix}
	 Q_{2,n}^{*[j]}(a)
	 & -Q_{1,n+1}^{*[j]}(a)
	 \\
	 j\big[ (j-1)P_{2,n}^{*[j-2]}(a) -  (b-a) P_{2,n}^{*[j-1]}(a) \big] 
	 & P_{1,n+1}^{*[j]}(a)
      \end{pmatrix}
      \begin{pmatrix}
	 Q_{2,n}^{*-1}(a) & 0
	 \\
	 0 & P_{1,n+1}^{*-1}(a)
      \end{pmatrix},
      \\[2ex]
      \widetilde C_{n+1}
      &= 
      \frac{1}{(n+1)!} 
      \begin{pmatrix}
	 0& -Q_{1,n+1}^{*[n+1]}(a)P_{1,n+1}^{*-1}(a) \\
	 (n+1) \big[ n P_{2,n}^{*[n-1]}(a) - (b-a) P_{2,n}^{*[n]}(a) \big]
	 Q_{2,n}^{*-1}(a)&
	 P_{1,n+1}^{*[n+1]}(a)P_{1,n+1}^{*-1}(a)
      \end{pmatrix},
      \\[2ex]
      \widetilde C_{n+2}
      &= \frac{1}{n!}
      \begin{pmatrix}
	 0&  0\\
	 P_{2,n}^{*[n]}(a)Q_{2,n}^{*-1}(a)&0
      \end{pmatrix}.
   \end{align*}
   Here $Q_{2,n}^{*\prime}(a)$ (resp. $Q_{2,n}^{*[j]}(a)$) denotes the first derivative (resp. $j$-th derivative) with respect to $z$.
\end{remark}

\begin{remark}\label{remsep4}
   The resolvent matrix $U^{(2n)}$ from \eqref{RM1} can be written
   as in \eqref{v2nD} where the coefficient matrices are for $1\le j\le n$ are
   \begin{align*}
      \widetilde D_0
      &=\begin{pmatrix}
	 I& \frac{1}{b-a}Q_{3,n}^*(a)P_{3,n}^{*-1}(a)\\
	 0&I
      \end{pmatrix},
      \\[2ex]
      \widetilde D_{j}
      &= \frac{1}{j!} 
      \begin{pmatrix}
	 Q_{4,n}^{*[j]}(a)Q_{4,n}^{*-1}(a)
	 & \frac{1}{(b-a)}Q_{3,n}^{*[j]}(a)P_{3,n}^{*-1}(a) 
	 \\
	 j P_{4,n}^{*[j-1]}(a)Q_{4,n}^{*-1}(a)
	 & \big[ P_{3,n}^{*[j]}(a)-\frac{j}{b-a}P_{3,n}^{*[j-1]}(a) \big] P_{3,n}^{*-1}(a)
      \end{pmatrix}
      \\
      &= \frac{1}{j!} 
      \begin{pmatrix}
	 Q_{4,n}^{*[j]}(a)
	 & Q_{3,n}^{*[j]}(a)
	 \\
	 j P_{4,n}^{*[j-1]}(a)
	 & (b-a) P_{3,n}^{*[j]}(a)- j P_{3,n}^{*[j-1]}(a)
      \end{pmatrix}
      \begin{pmatrix}
	 Q_{4,n}^{*-1}(a) & 0
	 \\
	 0 & \frac{1}{b-a} P_{3,n}^{*-1}(a)
      \end{pmatrix},
      \\[2ex]
      \widetilde D_{n+1}
      &= \frac{1}{n!} \begin{pmatrix}
	 0&  0
	 \\
	 P_{4,n}^{*[n]}(a)Q_{4,n}^{*-1}(a)&
	 -\frac{1}{(b-a)}P_{3,n}^{*[n]}(a)P_{3,n}^{*-1}(a)
      \end{pmatrix}.
   \end{align*}
   Here $Q_{4,n}^{*\prime}(a)$ (resp. $Q_{4,n}^{*[j]}(a)$) means the first derivative (resp. $j$-th derivative) with respect to $z$.
\end{remark}

\section*{Acknowledgment}

The research of Abdon E. Choque-Rivero was supported by 
CONACYT Project A1-S-31524 and CIC-UMSNH, Mexico.

\bibliography{moments}{}
\bibliographystyle{alpha}

\end{document}